\documentclass[10pt]{article}
\usepackage[left=1in,right=1in,top=1in,bottom=1in]{geometry}
\usepackage{hyperref}
\usepackage{amsmath, amsthm}
\usepackage{cleveref}
    \crefname{ex}{Example}{Examples}
    \crefname{thm}{Theorem}{Theorems}
\usepackage[margin=1.5cm]{caption}
\usepackage{multirow}
\usepackage{array}

\usepackage{tikz, graphicx}
\usepackage{subcaption}
	\usetikzlibrary{positioning}
	\usetikzlibrary{arrows}
        \tikzset{%
        fwdrxn/.style={very thick, arrows={-Stealth[length=5pt,width=5pt]}},
        revrxn/.style={very thick, arrows={-Stealth[length=5pt,width=5pt,left]}}
        }
	\newcommand{\tikzc}[1]
	    {\begin{center}\begin{tikzpicture}#1
	    \end{tikzpicture}\end{center}}
\usepackage{color, xcolor}
    
	\newcommand\blue[1]{{\textcolor{blue}{#1}}}
	\definecolor{orange}{RGB}{250, 140, 0}
		
	\definecolor{turq}{RGB}{0, 160, 160}
		
	\definecolor{violet}{RGB}{164, 98, 234}

\usepackage{amsthm}
	\newtheorem{thm}{Theorem}[section]
	\newtheorem{lem}[thm]{Lemma}

	\newtheorem*{theorem*}{Theorem}
	\theoremstyle{definition}
		\newtheorem{defn}[thm]{Definition}
		\newtheorem{ex}[thm]{Example}
	\newtheorem{rmk}[thm]{Remark}

	\newtheoremstyle{TheoremNum}
        {\topsep}{\topsep}              
        {\itshape}                      
        {}                              
        {\bfseries}                     
        {.}                             
        { }                             
        {\thmname{#1}\thmnote{ \bfseries #3}}
    \theoremstyle{TheoremNum}

\newcommand{\df}[1]{{\bf\emph{#1}}}		

\usepackage[normalem]{ulem}	

\usepackage{parskip}
    \makeatletter
    \def\thm@space@setup{%
        \thm@preskip=\parskip \thm@postskip=0pt
    }\makeatother
\usepackage{amsfonts, amssymb, amsrefs, mathtools}
\usepackage[nocompress]{cite} 
\newcommand{\eq}[1]{\begin{align*}#1\end{align*}}
	\newcommand{\eqn}[1]{\begin{align}#1\end{align}}  
\newcommand{\ds}{\displaystyle}	
\newcommand{\st}{\colon}                

\newcommand\mbf[1]{\mathbf{#1}}  
\newcommand\mrm[1]{\mathrm{#1}}
 
\newcommand{\rr}{\ensuremath{\mathbb{R}}}   
 
\newcommand{\zz}{\ensuremath{\mathbb{Z}}}

\renewcommand{\epsilon}{\varepsilon}	
\renewcommand{\phi}{\varphi}			

\DeclareMathOperator{\ran}{Im}	    	
\DeclareMathOperator{\Ran}{Im}		    %
\DeclareMathOperator{\Ker}{ker}		    
\DeclareMathOperator{\Span}{span}		
\newcommand{\braket}[2]{\langle{#1},\,{#2}\rangle}	
\newcommand{\kk}{\kappa}

\newcommand{\vv}[1]{{\boldsymbol{#1}}}  
\newcommand{\mm}[1]{\mathbf{#1}}               
\newcommand{\rrp}{\rr_{\geq 0}}
\newcommand{\rrpp}{\rr_{>0}}
\newcommand{\zzp}{\zz_{\geq 0}}

\newcommand{\RR}{\ensuremath{\rightleftharpoons}}
\newcommand{\FR}{\ensuremath{\rightarrow}}

\newcommand{\Gk}{\ensuremath{G_{\vv \kk}}}

\newcommand{\xx}{\vv x}
\newcommand{\yy}{\vv y}

\usepackage{enumitem}
\usepackage[version=4]{mhchem}
\usepackage{chemfig} 
    \newcommand{\cf}[1]{\chemfig{#1}}

\newcommand{\newt}{\mathrm{Newt}}
\newcommand{\Int}[1]{{#1}^{\mathrm{o}}}
\usepackage{authblk}
\title{
    Single-Target Networks
}
\author[1,2]{
         Gheorghe Craciun%
}
\author[1]{
        Jiaxin Jin%
}
\author[1]{
        Polly Y. Yu%
}
\affil[1]{\small Department of Mathematics, University of Wisconsin-Madison}
\affil[2]{\small Department of Biomolecular Chemistry, University of Wisconsin-Madison}
\date{} 

\begin{document}

\maketitle
\begin{abstract} 
\noindent
Reaction networks can be regarded as finite oriented graphs embedded in Euclidean space. \emph{\mbox{Single-target} networks} are reaction networks with an arbitrarily set of source vertices, but \emph{only one} sink vertex.   
    We completely characterize the dynamics of all mass-action systems generated by single-target networks, as follows: either  \emph{(i)} the system is globally stable for all choice of rate constants (in fact, is dynamically equivalent to a detailed-balanced system with a single linkage class) or \emph{(ii)} the system has no positive steady states for any choice of rate constants and all trajectories must converge to the boundary of the positive orthant or to infinity. Moreover, we show that global stability occurs if and only if the target vertex of the network is in the relative interior of the convex hull of the source vertices.
\end{abstract}

\section{Introduction}
\label{sec:intro}

Given a directed graph of reactions or interactions, one can write a system of differential equations modeling the time-dependent abundance of the interacting species, based on mass-action kinetics. The resulting dynamical systems are called \emph{mass-action systems}, and are very common models in chemistry, biochemistry and population dynamics~\cites{CraciunYu2018}. There has been a great amount of work on establishing  connections between the qualitative dynamics of these systems and their underlying network structures~\cites{HornJackson1972, Horn1972, Feinberg1972, Feinberg1987, Gunawardena2003, Angeli2009, FeinbergBook, CraciunYu2018}.

For example, if the underlying network is \emph{reversible} (i.e., for every edge, there is an edge in the reverse direction), then the mass-action system admits a positive steady state for any choice of positive rate constants~\cite{Boros2019}. In addition, if the rate constants satisfy some algebraic constraints such as the Wegscheider conditions~\cite{Wegscheider1901}, the mass-action system is in a state of thermodynamic equilibrium, where the rate of any forward reaction is balanced by the rate of the reverse reaction. Such a system, said to be \emph{detailed-balanced}, enjoys remarkable dynamical properties, like the existence of a globally defined Lyapunov function, and uniqueness of a positive steady state within every invariant polytope detemined by mass conservation laws.

Similarly, if the underlying network is \emph{weakly reversible} (i.e., every edge is part of an oriented cycle), again the mass-action system admits a positive steady state for any choice of positive rate constants~\cite{Boros2019}. If the rate constants satisfy some algebraic constraints, the mass-action system is \emph{complex-balanced}~\cites{Feinberg1972, Horn1972, HornJackson1972}, a generalization of detailed-balanced. Again, the system admits a globally defined Lyapunov function, has a unique positive steady state within every invariant polytope, and is conjectured to be globally stable. This is the \emph{Global Attractor Conjecture}, which has been proved in several cases: when the network has only one connected component~\cites{Anderson2011_GAC, BorosHofbauer2019}; when the system has dimension three or less~\cites{Pantea2012, CraciunNazarovPantea2013}, or when the network is \emph{strongly endotactic}~\cite{GopalkrishnanMillerShiu2014, classes_of_networks}. 

Some networks are always complex-balanced under mass-action kinetics, regardless of the values of rate constants: these are the weakly reversible network with deficiency zero~\cite{Feinberg1972, Horn1972}. One interpretation of the deficiency zero property is that the reaction vectors span the maximal dimensional subspace possible~\cite{FeinbergBook}.

In order to describe various properties of reaction networks, it is useful to visualize them in Euclidean space as \emph{Euclidean embedded graphs}~\cite{Craciun2019_TDI}. Each vertex of the network is naturally associated to a vector in $\rr^n$, via its stoichiometric coefficients; hence, every directed edge in the network (i.e., reaction) can be visualized as a vector between vertices of the network in $\rr^n$. The resulting directed graph in $\rr^n$ is called the Euclidean embedded graph of the reaction network, and its \emph{Newton polytope} is the convex hull of its \emph{source} vertices. A \emph{strongly endotactic network} is essentially an \lq\lq inward pointing\rq\rq\ one: any edge originating on the boundary of the Newton polytope must point inside the polytope or along its boundary (i.e., cannot point outside the polytope), and on any face of the polytope there exists an edge that starts on that face and points away from it.

Our main result concerns the class of \emph{single-target networks}. As the name suggests, these are reaction networks with exactly one sink vertex. In \Cref{thm:stnmas-stable,thm:stnmas} we prove that, under mass-action kinetics, a single-target network either has a globally stable positive steady state for any choice of positive rate constants, or has no positive steady state for any choice of rate constants. These results take advantage of the notion of \emph{dynamical equivalence} (\Cref{def:DE}), where different network structures can give rise to the same differential equations. More precisely, we prove that the dynamics generated by a single-target network is dynamically equivalent to a detailed-balanced system with one connected component if and only if the sink is in the relative interior of the Newton polytope, regardless of the choice of rate constants. 

In summary, the present work reveals a new class of reaction networks (i.e., {single-target networks}) for which the dynamics of the corresponding mass-action systems is completely determined by network structure, irrespective of parameter values (i.e., rate constants). As we mentioned above, the only other class of networks of this type are weakly reversible deficiency zero networks; on the other hand,  general single-target networks are neither weakly reversible nor deficiency zero. Moreover, we present examples of  reaction networks which, even though they are not single-target networks, are \emph{dynamically equivalent} to single-target networks under some mild assumptions; this allows us to use the theory of single-target networks  to characterize their dynamical properties.

This paper is organized as follows. After a preliminary section on mass-action systems and dynamical equivalence in \Cref{sec:crn}, we define single-target networks and prove our main results in \Cref{sec:stn}. For comparison, in \Cref{sec:multtarg} we consider networks with \emph{multiple targets}, and show that they may not be globally stable even if the sink vertices are contained in the interior of the convex hull of the source vertices. This last section suggests future directions to understanding the dynamics of strongly endotactic networks.

\section{Mass-action systems}
\label{sec:crn}

Throughout, let $\rrpp$ denote the positive real numbers, and $\rrpp^n$ denote the set of real vectors with positive components, i.e., $\xx \in \rrpp^n$ if $x_i > 0$ for all $i = 1,2,\dots, n$. We write $\xx > \vv 0$ when $\xx \in \rrpp^n$. Analogously, let $\rrp$, $\rrp^n$ denote the sets of non-negative numbers and vectors respectively. For any $\xx$, $\yy \in \rr^n$, define the vector operations 
    \eq{ 
        \xx^{\yy} &= x_1^{y_1}x_2^{y_2}\cdots x_n^{y_n}   \quad \text{ whenever $\xx \in \rrpp^n$},\\
        \log(\xx) &= (\log x_1, \log x_2, \ldots, \log x_n)^\top  \quad \text{ whenever $\xx \in \rrpp^n$}, \\
        \exp(\xx) &= ( e^{x_1}, e^{x_2}, \ldots, e^{x_n} )^\top, \\
        \xx \circ \yy &= (x_1y_1, x_2y_2, \ldots, x_ny_n)^\top,  
    }
and let $\braket{\xx}{\yy}$ denote the standard scalar product of $\rr^n$. If a set $X \subseteq \rr^n$ is contained in some affine subspace of $\rr^n$, we denote by $\Int{X}$  the \emph{relative interior} of  $X$ with respect to the usual topology of $\rr^n$. 

\medskip 

\begin{defn}
\label{def:crn}
	A \df{reaction network} is a directed graph $G = (V_G,E_G)$, where $V_G$ is a finite subset of $\rr^n$ and there are no self-loops. 
\end{defn}

When working with only one reaction network, we simply write $G = (V,E)$. An edge $(\yy, \yy')$ is also denoted $\yy \to \yy'$. If both $\yy\to\yy'$ and $\yy'\to\yy$ are edges, the \df{reversible pair} is denoted $\yy \RR \yy'$. A vertex $\yy \in V$ is a \df{source vertex} if $\yy \to \yy'$ is an edge in the network for some $\yy' \in V$; let $V_s \subseteq V$ denote the set of source vertices. A vertex $\yy' \in V$ is a \df{target vertex} if $\yy \to \yy' \in E$ for some $\yy \in V$.

Vertices are points in $\rr^n$, so an edge $\yy \to \yy' \in E$ can be regarded as a bona fide vector in $\rr^n$. Each edge is associated to its {\df{reaction vector}} $\yy' - \yy \in \rr^n$.

We will construct a dynamical system using the graph $G$ and the data stored in the vertices. The coordinates of a source vertex are exponents of a monomial. In algebra, the Newton polytope of a polynomial is the convex hull of the exponents of the monomials. Here, we define the Newton polytope using all the monomials appearing in the right-hand side of the dynamical system. In \cite{GopalkrishnanMillerShiu2014}, the Newton polytope of a reaction network is also called a \emph{reactant polytope}.

\begin{defn}
\label{def:newton}
    The \df{Newton polytope} of a reaction network $G = (V,E)$ is the convex hull of the source vertices, i.e.,
    \eq{ 
        \newt(G) = \left\{ 
            \sum_{\yy \in V_s} \alpha_{\yy} \yy 
            \,\colon\, 
            \alpha_{\yy} \geq 0 \text{ and } \sum_{\yy \in V_s} \alpha_{\yy} = 1
        \right\}.
    }
\end{defn}
    The more important object is the relative interior of the Newton polytope
    \eq{ 
        \Int{\newt(G)} = \left\{ 
            \sum_{\yy \in V_s} \alpha_{\yy} \yy 
            \,\colon\, 
            \alpha_{\yy} > 0 \text{ and } \sum_{\yy \in V_s} \alpha_{\yy} = 1
        \right\}.
    }
    Note that in $\Int{\newt(G)}$, all the coefficients in the sum must be positive.

\begin{defn}
\label{def:mas}
    Let $G = (V,E)$ be a reaction network in $\rr^n$ with edge set $E = \{\yy_i \to \yy'_i \}_{i=1}^R$. Let $\vv\kk = (\kk_i)_{i=1}^R$ be a vector of positive constants, called the vector of \df{rate constants}. A graph $G$ together with a vector of rate constants $\vv\kk$ gives rise to a \df{mass-action system}%
        \footnote{%
            Strictly speaking, because we have taken $V \subseteq \rr^n$ (instead of $V \subseteq \rrp^n$ or $V \subseteq \zzp^n$), the weighted directed graph $\Gk$ may \emph{not} define a mass-action system as it is classically understood; rather it is a \emph{power-law system}~\cite{CraciunNazarovPantea2013}. However, in this work, we are interested in a class of systems which shares the same dynamics as complex-balanced mass-action systems even though $V \subseteq \rr^n$. For simplicity, we refer to these as \emph{mass-action systems}.   
        },
    denoted $\Gk$. Its \df{associated dynamical system} is the system of differential equations on $\rrpp^n$ given by
    \eqn{\label{eq:mas}
        \frac{d\xx}{dt} = \sum_{i=1}^R \kk_i \xx^{\yy_i}(\yy'_i - \yy_i).
    }
\end{defn}

When the edge set is not ordered, we sometimes index the rate constants using the edge label itself, e.g., $\kk_{\yy\to\yy'}$ is the rate constant of the edge $\yy \to \yy'$. It is sometimes convenient to refer to $\kk_{\yy \to \yy'}$ even though $\yy \to \yy'$ may \emph{not} be an edge in the network. In such cases, the convention is to take $\kk_{\yy\to\yy'} = 0$.

The system of differential equations \eqref{eq:mas} can be written as 
    \eq{ 
        \frac{d\xx}{dt} = \mm \Gamma \begin{pmatrix}
            \kk_1 \xx^{\yy_1} \\
            \kk_2 \xx^{\yy_2} \\
            \vdots \\
            \kk_R \xx^{\yy_R}
        \end{pmatrix},
    }
where the \df{stoichiometric matrix} $\mm \Gamma$ has as its $i$th column the reaction vector $\yy'_i - \yy_i$. Since $\frac{d\xx}{dt}$ lies in the \df{stoichiometric subspace} $S = \ran\mm\Gamma$, the solution to the system \eqref{eq:mas}, with initial value $\xx_0 \in \rrpp^n$, lies in the affine space $\xx_0+S$. The \df{positive stoichiometric compatibility class}\footnote{Note that for classical mass-action systems the positive orthant is forward invariant~\cite{VOLPERT}, which implies that each positive stoichiometric compatibility class is forward invariant. For  {\em power-law} systems this is not true in general; nevertheless, for  {\em single-target} systems for which the target vertex is in the interior of the convex hull of its source vertices (which are the main focus of this paper) we  show that they are dynamically equivalent to reversible single-linkage-class systems, which implies that they are permanent~\cite{GopalkrishnanMillerShiu2014}, and in turn this  {\em does} imply that the positive orthant is forward invariant for such single-target systems.} is the set $(\xx_0+S)_> = (\xx_0+S)\cap \rrpp^n$.

We say that a reaction network $G = (V,E)$ is \df{reversible} if $\yy' \to \yy \in E$ whenever $\yy \to \yy' \in E$, and is  \df{weakly reversible} if every connected component of $G$ is strongly connected, i.e., every edge is part of an oriented cycle. A mass-action system $\Gk$ is called  reversible or weakly reversible if the underlying reaction network $G$ is reversible or weakly reversible. 

\subsection{Detailed-balanced systems}

Reversibility and weak reversibility are graph-theoretic properties necessary for \emph{detailed-balanced} and \emph{complex-balanced} steady states, respectively. These steady states are known to be asymptotically stable and conjectured to be globally stable. The Global Attractor Conjecture has been proved for certain classes of mass-action systems~\cites{GopalkrishnanMillerShiu2014, Pantea2012, CraciunNazarovPantea2013}, including the case where the network has a single connected component~\cites{Anderson2011_GAC, BorosHofbauer2019}. A general proof has been proposed in \cites{Craciun_GAC, Craciun2019_TDI}. 

\begin{defn}
\label{def:steadystate}
    Let $\Gk$ be a mass-action system. 
    \begin{enumerate}[label=(\alph*)]
    \item 
        A state $\xx^* > \vv 0$ is a \df{positive steady state} if the right-hand side of 
        \Cref{eq:mas} evaluated at $\xx^*$ is $\vv 0$. 
    \item 
        A positive steady state $\xx^* > \vv 0$ is \df{detailed-balanced} if for every $\yy \to \yy' \in E$, we have 
        \eqn{ \label{eq:DB} 
            \kk_{\yy\to \yy'} (\xx^*)^{\yy} = \kk_{\yy' \to \yy} (\xx^*)^{\yy'}.
        }
    \item 
        A positive steady state $\xx^* > \vv 0$ is \df{complex-balanced} if for every $\yy_0 \in V$, we have 
        \eqn{ \label{eq:CB} 
            \sum_{\yy_0 \to \yy' \in E} \kk_{\yy_0\to \yy'} (\xx^*)^{\yy_0} 
            = 
            \sum_{\yy \to \yy_0 \in E} \kk_{\yy \to \yy_0} (\xx^*)^{\yy}.
        }
    \end{enumerate}
\end{defn}

If $\Gk$ admits a detailed-balanced steady state, the network is necessarily reversible; \Cref{eq:DB} balances the fluxes flowing across a reversible pair of edges. If $\Gk$ admits a complex-balanced steady state, the network is weakly reversible~\cite{HornJackson1972}; \Cref{eq:CB} balances the net flux  across a vertex of the graph. A detailed-balanced steady state is also complex-balanced. 

If a mass-action system has a complex-balanced steady state $\xx^*$, then all its positive steady states are complex-balanced~\cite{HornJackson1972}. Its set of positive steady states can be represented as $E_{\vv\kk} = \{ \xx > \vv 0 \colon \ln(\xx) - \ln(\xx^*) \in S\}$, where $S$ is the stoichiometric subspace. Moreover, there is a globally defined Lyapunov function~\cite{HornJackson1972} 
    \eq{ 
        V(\xx) = \sum_i x_i(\ln x_i - \ln x_i^* - 1).
    }
We refer to a mass-action system with a complex-balanced steady state as a \df{complex-balanced system}. The above is also true for detailed-balanced steady states, and \df{detailed-balanced systems}.

In general, a weakly reversible mass-action system may not have a complex-balanced steady state --- similarly for reversible systems and detailed-balanced steady states --- unless the rate constants satisfy additional algebraic constraints~\cites{CraciunDickensteinSturmfelsShiu2009, DickensteinPerezmillan2011, Feinberg1989, Onsager1931, Wegscheider1901, SchusterSchuster1989}. A non-negative integer called the \emph{deficiency} of the network specifies the number of independent algebraic constraints on the rate constants that are necessary and sufficient for a weakly reversible mass-action system to be complex-balanced~\cite{CraciunDickensteinSturmfelsShiu2009}.

\begin{defn}
\label{def:deficiency}
    Let $G$ be a reaction network with $m$ vertices, $\ell$ connected components, and stoichiometric subspace $S$. The \df{deficiency} of the network $G$ is the non-negative integer $\delta = m - \ell - \dim S$. 
\end{defn}

Detailed-balancing, being more restrictive than complex-balancing, requires that the rate constants satisfy the algebraic conditions for complex-balancing, in addition to the \emph{circuit conditions}: for every cycle in the reversible network, the product of rate constants in one direction equals that of the other direction~\cite{DickensteinPerezmillan2011}. In other words, suppose in one orientation of a cycle, the rate constants are $\kk_{1+}$, $\kk_{2+}, \ldots, \kk_{r+}$, and in the other orientation, the rate constants are $\kk_{1-}$, $\kk_{2-}, \ldots, \kk_{r-}$; then the circuit condition along this cycle is
    \eqn{ 
    \label{eq:circuit}
        \prod_{i=1}^r \kk_{i+} = \prod_{i=1}^r \kk_{i-}.
    }

For a reversible system, the algebraic conditions for detailed-balance are also not difficult to state~\cite{DickensteinPerezmillan2011}\footnote{%
    See \cite{Feinberg1989} for a different, but equivalent, set of conditions.
}. %
Choose a forward direction for each reversible pair and let $\kk_{i+}$ be its rate constant; let $\kk_{i-}$ be the rate constant of the backward direction. Suppose the network has $p$ reversible pairs of edges. Let $\mm \Gamma' \in \rr^{n\times p}$ be the matrix whose columns are the reaction vectors of the forward directions.

\begin{thm}
\label{thm:DB-Weg}
    The reversible mass-action system $\Gk$ is detailed-balanced if and only if every $\vv J \in \ker\mm\Gamma' \subseteq \rr^{p}$ satisfies the Wegscheider condition:
        \eq{ 
            \prod_{i=1}^p (\kk_{i+})^{J_i} = \prod_{i=1}^p (\kk_{i-})^{J_i}. 
        }
\end{thm}

\subsection{Dynamical equivalence}

A recurring theme in detailed-balanced or complex-balanced system is that dynamical behaviour can sometimes be deduced or ruled out by network structure. When studying mass-action kinetics, it is the associated system of differential equations that is of interest. One might ask whether it is possible to attach different network properties to the same system. There are several approaches of attaching a dynamical system to networks, particularly networks with some desirable properties. For example, \emph{network translation} searches for a weakly reversible network that generate a given dynamical system using \emph{generalized mass-action kinetics}~\cite{Johnston2014, JohnstonBurton2019, MuellerRegensburger2012}. Some weakly reversible generalized mass-action systems enjoy the algebraic properties of a  complex-balanced system. Another method, staying within the realm of mass-action systems, is that of \emph{dynamical equivalence}. The associated dynamical system \eqref{eq:mas} of a mass-action system $\Gk$ is uniquely defined; however, different reaction networks can give rise to the same system of differential equations under mass-action kinetics~\cite{cpIdentifiability, CommentOnID, CraciunJinYu2019}. We say that a system of differential equations $\dot{\xx} = \vv f(\xx)$ \df{can be realized} by a network $G$ if there exists a vector of rate constants $\vv\kk > \vv 0$ such that is associated dynamical system of $\Gk$ is $\dot{\xx} = \vv f(\xx)$.

\begin{defn}
\label{def:DE}
    Two mass-action systems $\Gk$ and $G'_{\vv \kk'}$ are \df{dynamically equivalent} if their associated dynamical systems agree on all of $\rrpp^n$. Equivalently, for every vertex $\yy_0 \in V_{G} \cup V_{G'}$, we have 
    \eqn{ \label{eq:DE2} 
        \sum_{\yy_0 \to \yy \in E_{G}} \kk_{\yy_0\to\yy} (\yy - \yy_0)
        = 
        \sum_{\yy_0 \to \yy \in E_{G'}} \kk'_{\yy_0\to\yy} (\yy - \yy_0).
    }  
\end{defn}

In going from the differential equations to the condition \eqref{eq:DE2}, we used the linear independence of the monomials. This approach divorces the non-linearity from the linear part of the problem. Another notion that divorces the non-linear from the linear is that of fluxes. The \emph{flux} of an edge $\yy \to \yy'$ in a mass-action system $\Gk$ is the positive quantity $\kk_{\yy \to \yy'} \xx^{\yy}$ where $\xx > \vv 0$ is a given state. In what follows, $\rrpp^E$ denote the set of vectors with positive coordinates and indexed by $E$. 

\begin{defn}
\label{def:flux}
    A \df{flux vector} $\vv J \in \rrpp^E$ on a reaction network $G = (V,E)$ is a vector of \emph{positive} numbers for each edge. If $\vv J \in \Ker \mm\Gamma \cap \rrpp^E$, then $\vv J$ is a \df{steady state flux}.
\end{defn}

Dynamical equivalence, ultimately a linear property, can be checked using fluxes. If $\Gk$ and $G'_{\vv\kk'}$ are dynamically equivalent, then for every $\yy_0 \in V_G \cup V_{G'}$, we can multiply \eqref{eq:DE2} by $\xx^{\yy_0}$ at any positive state $\xx \in \rrpp^n$. The fluxes $J_{\yy_0 \to \yy} = \kk_{\yy_0\to\yy}\xx^{\yy_0}$ and $J'_{\yy_0 \to \yy} = \kk'_{\yy_0\to\yy}\xx^{\yy_0}$ defined on $G$ and $G'$ respectively satisfy
    \eqn{ \label{eq:FE2} 
        \sum_{\yy_0 \to \yy \in E_{G}} J_{\yy_0\to\yy} (\yy - \yy_0)
        = 
        \sum_{\yy_0 \to \yy \in E_{G'}} J'_{\yy_0\to\yy}  (\yy - \yy_0)
    }  
for every $\yy_0 \in V_G \cup V_{G'}$. We say that the two are \emph{flux equivalent}. Indeed, two mass-action systems are dynamically equivalent if and only if the corresponding fluxes at an arbitrary state satisfy \eqref{eq:FE2} for every vertex. See \cite{CraciunJinYu2019} for the correspondence between mass-action systems and fluxes on a network.

Ultimately dynamical equivalence and flux equivalence are linear feasibility problems. There exist algorithms based on linear programming that search for dynamically equivalent realizations with certain properties, e.g., complex-balanced or detailed-balanced~\cite{SzederkenyiHangos2011}; weakly reversible or reversible~\cite{RudanSzederkenyiKatalinPeni2014}; with minimal deficiency~\cite{JohnstonSiegelSzederkenyi2013, LiptakSzederkenyiHangos2015}. An implementation of some of these algorithms is available as a MATLAB toolbox~\cite{Szederkenyi2012_Toolbox}. Although these algorithms require a predetermined set of vertices as input, in the case of detailed-balanced, complex-balanced, reversible or weakly reversible realizations, it suffices to use the exponents of the monomials in the differential equations~\cite{CraciunJinYu2019}.

\section{Single-target networks}
\label{sec:stn}

In this section, we classify all single-target networks under mass-action kinetics: those that have a globally attracting positive steady state for all choices of positive rate constants, and those that have no positive steady state for any choice of rate constants. The former occurs if and only if the target is in the relative interior of the Newton polytope, the convex hull of the source vertices.  

It is not difficult to show that if every reaction vector points to the relative interior of the Newton polytope (i.e., \lq\lq inward pointing\rq\rq), then the mass-action system is always dynamically equivalent to a weakly reversible system. It follows immediately that the system has a positive steady state~\cite{Boros2019} and is conjectured to be \emph{permanent}~\cite{CraciunNazarovPantea2013}. In the case of a single-target network with \lq\lq inward pointing\rq\rq\ reaction vectors (to be made precise below), we show that the dynamics is essentially that of a detailed-balanced system.  

\begin{defn}
\label{def:stn}
	A reaction network $G = (V,E)$ is a {\df{single-target network}} if there exists a vertex $\vv y^*$ such that $V \setminus \{ \vv y^*\}$ is the set of source vertices, and $E = \{\vv y \FR \vv y^* : \vv y \in V\setminus\{ \vv y^*\} \}$. We call $\vv y^*$ the \df{target vertex}, while the remaining vertices are \df{source vertices}.
\end{defn}

\begin{figure}[h!tbp]
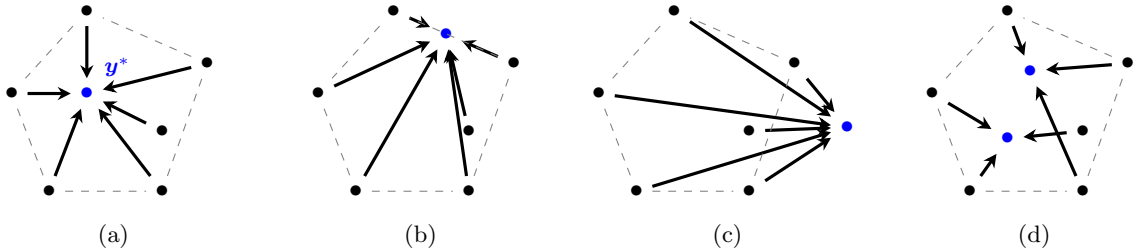

\centering
\begin{subfigure}[b]{0.24\textwidth}
    \centering 
    \input{fig-STN_ex-a}
    \caption{}
    \label{fig:exstn-a}
\end{subfigure}
\begin{subfigure}[b]{0.24\textwidth}
    \centering 
    \input{fig-STN_ex-b}
    \caption{} 
    \label{fig:exstn-b}
\end{subfigure}
\begin{subfigure}[b]{0.24\textwidth}
    \centering 
    \input{fig-STN_ex-c}
    \caption{} 
    \label{fig:exstn-c}
\end{subfigure}
\begin{subfigure}[b]{0.24\textwidth}
    \centering 
    \input{fig-STN_ex-d}
    \caption{} 
    \label{fig:exstn-d}
\end{subfigure}
    \caption{(a) A single-target network that is globally stable under mass-action kinetics. (b)--(c) Single-target networks with no positive steady states. (d) Not a single-target network.}
    \label{fig:exstn}
\end{figure}

\begin{ex}
\label{ex:stn}
	The reaction networks (a)--(c) in \Cref{fig:exstn} are single-target networks, while (d) is not a single-target network. The target vertex of (a) is in the relative interior of its Newton polytope. We will show that network (a) is typical of single-target networks that have exactly one globally stable steady state within each stoichiometric compatibility class, while the networks (b) and (c) have no positive steady state, regardless of the choice of kinetics. The deficiencies of the networks (a)--(c) are $\delta = 6- \dim S$, while that of (d) is $\delta = 7- \dim S$, where $S$ is the stoichiometric subspace. 
\end{ex}

The geometry of a single-target network, i.e., whether the target is in the relative interior of the Newton polytope, determines whether the network admits a steady state flux, which is necessary for the existence a positive steady state under reasonable kinetics. In particular, the geometry can rule out the existence of positive steady states. 

\begin{lem}
\label{lem:stnFluxDB}
	Let $G$ be a single-target network. There exists a steady state flux on $G$ if and only if the target vertex is in the relative interior of its Newton polytope. 
\end{lem}
\begin{proof}
Let $\vv y^*$ denote the target vertex of $G$, and enumerate the source vertices as $\vv y_1, \vv y_2,\dots, \vv y_m$. The vector $\vv J = (J_i)_{\vv y_i \FR \vv y^* \in E} \in \rrpp^E$ is a steady state flux if and only if 
    \eq{ 
        \sum_{i=1}^m J_{i} (\vv y^* - \vv y_i ) = \vv 0.
    }
Rearranging, we see that $\yy^* = \sum_i \frac{J_i}{J_T}  \yy_i$, where $J_T = \sum_i J_i$, and each $J_i > 0$. By definition, $\yy^* \in \Int{\newt(G)}$. 
\end{proof}

\begin{rmk}
\label{rmk:notinterior}
At first glance, \Cref{lem:stnFluxDB} is a result about fluxes, with no reference to any underlying kinetics. However, suppose the flux vector arises from any reasonable kinetics, such as mass-action or Michaelis--Menten kinetics --- indeed the argument holds if each reaction rate function is differentiable (or Lipschitz) function mapping a state in $\rrpp^n$ to a positive number. Then for a single-target network whose target vertex is outside the relative interior of the Newton polytope, we can show that the trajectory, starting from any positive initial condition, will simply converge to the boundary of the positive orthant or to infinity. Indeed, there exists a Lyapunov function for such a dynamical system. When the target vertex is not in the relative interior of the convex hull of the sources, i.e.,  $\vv y^* \not\in \Int{\newt(G)}$, geometrically there is a hyperplane $H$ (within the stoichiometric subspace) such that all the reaction vectors $\{ \yy^* - \yy\}$ lie in a halfspace defined by $H$. (See \Cref{fig:exstn-b,fig:exstn-c} for examples of such networks.) Let $\vv w$ be orthogonal to $H$ such that $\braket{\vv w}{\vv y^* - \vv y} \leq 0$ for all reactions $\vv y \to \vv y^*$. Then $V(\xx) = \braket{\xx}{\vv w}$ defines a linear Lyapunov function for the single-target system, and all trajectories must converge to the boundary of the positive orthant or to infinity. 
\end{rmk}

Even with the target vertex in $\Int{\newt(G)}$, to deduce a positive steady state from a steady state flux $\vv J$ involves finding a positive solution $\xx$ to the non-linear equations $J_{\yy \to \yy'} = k_{\yy \to \yy'}\xx^{\yy}$ for every reaction $\yy \to \yy'$. We will prove the existence of steady state for such single-target mass-action systems in \Cref{thm:stnmas}. The result also applies to systems that are dynamically equivalent to a single-target network; for example see \Cref{ex:dense4,ex:dense5}. Our proof of the existence and global stability of a positive state will make use of the following theorems.

\begin{thm}[\cites{PachterSturmfels, Birch1963, HornJackson1972}]
\label{thm:Birch}
	Let $S \subseteq \rr^n$ be a vector subspace, and let $\vv x_0$, $\vv x^* \in \rrpp^n$ be two arbitrary positive vectors. The intersection $(\vv x_0 + S)\cap (\vv x^* \circ \exp S^\perp)$ consists of exactly one point, where 
	    $\vv x^* \circ \exp S^\perp = \{ \vv x^* \circ \exp(\vv s) \st \vv s \in S^\perp\}$. 
\end{thm}
\begin{thm}[\cites{Anderson2011_GAC, BorosHofbauer2019}]
\label{thm:GAC1lnk}
    Let $G_{\vv\kk}$ be a complex-balanced system with one connected component. Any positive steady state is a global attractor within its stoichiometric compatibility class.
\end{thm}

We now give a necessary and sufficient condition for a single-target network to be dynamically equivalent to a detailed-balanced system under mass-action kinetics. This result is related to the theory of \emph{star-like networks}~\cite{FeinbergBook}, which have been shown to have a unique asymptotically stable steady state within each stoichiometric compatibility class. In what follows, $\rrpp^E$ denote the set of vectors of rate constants, indexed by $E$.

\begin{thm}
\label{thm:stnmas-stable}
    Let $G = (V,E)$ be a single-target network whose target vertex is in the relative interior of the Newton polytope. Then for any vector of rate constants $\vv \kk \in \rrpp^{E}$, the mass-action system $G_{\vv \kk}$ is dynamically equivalent to a detailed-balanced system that has a single connected component. 
\end{thm}
\begin{proof}
	Let $\vv y^*$ denote the target vertex, and enumerate the source vertices $\vv y_1,\vv y_2,\dots, \vv y_m$. Let $\mm\Gamma \in \rr^{n \times m}$ be the stoichiometric matrix, whose $j$th column is the reaction vector $\vv y^* - \vv y_j$. 
	Let $\kk_j > 0$ be an arbitrary rate constant for the edge $\vv y_j \FR \vv y^*$, and let $\vv\kk = (\kk_j)_{j=1}^m$. Recall that the relative interior of the Newton polytope is 
    \eq{ 
	    \Int{\newt(G)} = \left\{ 
	    \sum_{j=1}^m \alpha_j \yy_j :
	        \alpha_j > 0
	         \text{ and } 
	        \sum_{j=1}^m \alpha_j = 1 
	    \right\}. 
	}
	
	We want to prove that $G_{\vv \kk}$ is dynamically equivalent to a detailed-balanced system with vertex set $V_{G'} = V_G$ and edge set $E_{G'} = E_G \cup \{ \yy^* \to \yy_j\}_{j=1}^m$. Moreover, for the original edges $\yy_j \to \yy^*$, we keep the same rate constants $\kk_j$. Let $\kk'_j$ denote the rate constant of the reversible edge $\yy^* \to \yy_j$, whose value is to be determined. Consider the following conditions with unknowns $\kk'_j > 0$ and $\xx \in \rrpp^n$: 
	\eqn{
		&  \sum_{j=1}^m \kk'_j (\vv y_j - \vv y^*) = \vv 0 , 
		    \label{pf:DE}
		\\
		&
		\kk_j \vv x^{\vv y_j} = \kk'_j \vv x^{\vv y^*}
			\qquad \text{for all $1 \leq j\leq m$}.
			\label{pf:DB}
	}
    The condition \eqref{pf:DE} ensures that the resulting system is dynamically equivalent to the original since the only difference between the two networks are the edges with source $\yy^*$. The condition \eqref{pf:DB} ensures that resulting system is a detailed-balanced system with positive steady state $\xx$. Condition \eqref{pf:DE} can be replaced with $\vv \kk' = (\kk'_j)_{j=1}^m \in \Ker \mm\Gamma$. Isolating $\kk_j'$ in condition \eqref{pf:DB}, we obtain
	\eq{
		\kk'_j &= \kk_j \vv x^{\vv y_j - \vv y^*}
		= \kk_j e^{\braket{\vv y_j -\vv y^*}{  \log \vv x} }.
	} 
    So \eqref{pf:DB} is equivalent to $\vv \kk' \in \vv \kk \circ \exp(\Ran \mm\Gamma^\top)$. 
    Therefore, that $\Gk$ is dynamically equivalent to a detailed-balanced system follows from the existence of $\vv\kk'$ in the intersection $\Ker \mm\Gamma \cap (\vv \kk \circ \exp(\Ran \mm\Gamma^\top)) \subseteq \rrpp^{m}$. 
	
	By \Cref{lem:stnFluxDB}, there exists a steady state flux $\vv J$ on $G$, i.e., $\vv J \in \Ker \mm\Gamma \cap \rrpp^m$. Hence, 
	\eq{
	    \Ker \mm\Gamma \cap (\vv \kk \circ \exp(\Ran \mm\Gamma^\top))
	    &= (\vv J + \Ker \mm\Gamma) \cap (\vv \kk \circ \exp(\Ker \mm\Gamma^\perp)),
	}
    which is guaranteed to be non-empty for any positive $\vv J$, $\vv \kk$ by \Cref{thm:Birch}~\cites{Birch1963, HornJackson1972}. Let $\vv\kk' = (\kk'_j)_{j=1}^m$ be in the intersection. Therefore, there exist positive solutions $\xx \in \rrpp^n$ and $\kk_j > 0$ satisfying conditions \eqref{pf:DE}--\eqref{pf:DB}. The graph $G'$ consists of the original edges $\yy_j \to \yy^*$ with the original rate constants $\kk_j > 0$ and the edges $\yy^* \to \yy_j$ with rate constants $\kk'_j>0$. In other words, $G_{\vv\kk}$ is dynamically equivalent to a detailed-balanced system $G'_{\tilde{\vv\kk}}$, where $G'$ is strongly connected and $\tilde{\vv\kk} \in \rrpp^{E'}$ has coordinates given by $\vv\kk$ and $\vv\kk'$.
\end{proof}

\begin{thm}
\label{thm:stnmas}
	Let $G = (V,E)$ be a single-target network. For any vector of rate constants $\vv \kk \in \rrpp^{E}$, let $G_{\vv \kk}$ denote the corresponding  mass-action system. Then exactly one of the following is true.
	\begin{enumerate}
	\item
		For any $\vv \kk$, the mass-action system $G_{\vv \kk}$ has no positive steady states and all trajectories must converge to the boundary of the positive orthant or to infinity.
	\item
		For any $\vv \kk$, the mass-action system $G_{\vv \kk}$ has exactly one positive steady state within each of its stoichiometric compatibility class. Furthermore, this steady state is globally stable within its class.
	\end{enumerate} 
The latter occurs if and only if the target vertex of $G$ is in the relative interior of the Newton polytope.
\end{thm}
\begin{proof}
	If $\vv y^* \not\in \Int{\newt(G)}$, by \Cref{lem:stnFluxDB} and \Cref{rmk:notinterior} the network $G$ admits no positive flux vector, i.e., $\Ker \mm\Gamma \cap \rrpp^E = \emptyset$; therefore, any mass-action system generated by $G$ cannot have a positive steady state and all trajectories must converge to the boundary of the positive orthant or to infinity. However, if $\vv y^* \in \Int{\newt(G)}$, then by \Cref{thm:stnmas-stable} the mass-action system is dynamically equivalent to a detailed-balanced system with one connected component regardless of the choice of rate constants. Since detailed-balanced systems are complex-balanced, this system, with one connected component, has within each of its stoichiometric compatibility class exactly one positive steady state, which is globally stable, as stated in \Cref{thm:GAC1lnk}~\cites{Anderson2011_GAC, BorosHofbauer2019}.
\end{proof}

\begin{ex}
	Consider the single-target networks in \hyperref[fig:exstn]{Figures 1(a)--(c)}. Mass-action systems generated by networks (b) and (c) can never have positive steady states, while systems generated by the network (a) would not only have exactly one positive steady state within every stoichiometric compatibility class, but that the steady state is globally stable within its stoichiometric compatibility class.
\end{ex}

\begin{figure}[h!tb]
  \centering\setlength\tabcolsep{1.5em}
  \begin{tabular}{ccc}
    \multirow{2}{*}[0.9cm]{\subcaptionbox{\label{fig-tri-full}}{
        \begin{tikzpicture}[scale=1]
            \draw [step=1, gray, very thin] (0,0) grid (2.75,2.5);
    \draw [ ->, black!70!white] (0,0)--(2.75,0);
    \draw [ ->, black!70!white] (0,0)--(0,2.5);

    \node (3) at (2,0) {$\bullet$}; 
    \node (2) at (1,1) {$\bullet$}; 
    \node (1) at (0,2) {$\bullet$}; 
    
    \node [left =-3pt of 1] {$\yy_1$}; 
    \node [above right =-5pt of 2] {$\yy_2$}; 
    \node [below =-3pt of 3] {$\yy_3$}; 
    
    \draw [fwdrxn, transform canvas={xshift=-1pt,yshift=-1pt}] (1) -- (2) ;
    \draw [fwdrxn, transform canvas={xshift=1pt,yshift=1pt}] (2) -- (1) ;
    \draw [fwdrxn, transform canvas={xshift=-1pt,yshift=-1pt}] (2) -- (3) ;
    \draw [fwdrxn, transform canvas={xshift=1pt,yshift=1pt}] (3) -- (2) ;
    \draw [fwdrxn, transform canvas={xshift=-3pt,yshift=-3pt}] (3) -- (1) ;
    \draw [fwdrxn, transform canvas={xshift=3pt,yshift=3pt}] (1) -- (3) ;
        \end{tikzpicture}
    }}
    & \subcaptionbox{\label{fig-triA1}}{
        \begin{tikzpicture}[scale=0.7]
            \draw [step=1, gray, very thin] (0,0) grid (2.75,2.5);
    \draw [ ->, black!70!white] (0,0)--(2.75,0) node [right] {};
    \draw [ ->, black!70!white] (0,0)--(0,2.5);

    \node [inner sep=0pt, outer sep=0pt] (3) at (2,0) {$\bullet$}; 
    \node [inner sep=0pt, outer sep=0pt] (2) at (1,1) {$\bullet$}; 
    \node [inner sep=0pt, outer sep=0pt] (1) at (0,2) {$\bullet$};

    \draw [thick, arrows={-Stealth[length=5pt,width=5pt]}, red!50!black] (1)--(0.35,1.65);
    \draw [thick, arrows={-Stealth[length=5pt,width=5pt]}, red!50!black] (2)--(1.35,0.65);
    \draw [thick, arrows={-Stealth[length=5pt,width=5pt]}, red!50!black] (3)--(1.65,0.35);
        \end{tikzpicture}
        }
    & \subcaptionbox{\label{fig-triA2}}{
        \begin{tikzpicture}[scale=0.7]
            \draw [step=1, gray, very thin] (0,0) grid (2.75,2.5);
    \draw [ ->, black!70!white] (0,0)--(2.75,0) node [right] {};
    \draw [ ->, black!70!white] (0,0)--(0,2.5);

    \node [inner sep=0pt, outer sep=0pt] (3) at (2,0) {$\bullet$}; 
    \node [inner sep=0pt, outer sep=0pt] (2) at (1,1) {$\bullet$}; 
    \node [inner sep=0pt, outer sep=0pt] (1) at (0,2) {$\bullet$}; 
    
    \node [inner sep=0pt, outer sep=0pt, blue] (t) at (1.5,0.5) {$\bullet$}; 
    
    
    \draw [thick, arrows={-Stealth[length=5pt,width=5pt]}, transform canvas={xshift=-1pt,yshift=-1pt}] (2) -- (t) ;
    \draw [thick, arrows={-Stealth[length=5pt,width=5pt]}, transform canvas={xshift=-1pt,yshift=-1pt}] (3) -- (t) ;
    \draw [thick, arrows={-Stealth[length=5pt,width=5pt]}, transform canvas={xshift=2pt,yshift=2pt}] (1) -- (t) ;
        \end{tikzpicture}
        }
    \\
    & \subcaptionbox{\label{fig-triB1}}{
        \begin{tikzpicture}[scale=0.7]
            \draw [step=1, gray, very thin] (0,0) grid (2.75,2.5);
    \draw [ ->, black!70!white] (0,0)--(2.75,0) node [right] {};
    \draw [ ->, black!70!white] (0,0)--(0,2.5);

    \node [inner sep=0pt, outer sep=0pt] (3) at (2,0) {$\bullet$}; 
    \node [inner sep=0pt, outer sep=0pt] (2) at (1,1) {$\bullet$}; 
    \node [inner sep=0pt, outer sep=0pt] (1) at (0,2) {$\bullet$}; 
    
    \draw [thick, arrows={-Stealth[length=5pt,width=5pt]}, red!50!black] (1)--(0.35,1.65);
    \draw [thick, arrows={-Stealth[length=5pt,width=5pt]}, red!50!black] (2)--(0.65,1.35);
    \draw [thick, arrows={-Stealth[length=5pt,width=5pt]}, red!50!black] (3)--(1.65,0.35);
        \end{tikzpicture}
        }
    & \subcaptionbox{\label{fig-triB2}}{
        \begin{tikzpicture}[scale=0.7]
            \draw [step=1, gray, very thin] (0,0) grid (2.75,2.5);
    \draw [ ->, black!70!white] (0,0)--(2.75,0) node [right] {};
    \draw [ ->, black!70!white] (0,0)--(0,2.5);

    \node [inner sep=0pt, outer sep=0pt] (3) at (2,0) {$\bullet$}; 
    \node [inner sep=0pt, outer sep=0pt] (2) at (1,1) {$\bullet$}; 
    \node [inner sep=0pt, outer sep=0pt] (1) at (0,2) {$\bullet$}; 
    
    \node [inner sep=0pt, outer sep=0pt, blue] (t) at (0.5,1.5) {$\bullet$}; 
    
    
    \draw [thick, arrows={-Stealth[length=5pt,width=5pt]}, transform canvas={xshift=1pt,yshift=1pt}] (2) -- (t) ;
    \draw [thick, arrows={-Stealth[length=5pt,width=5pt]}, transform canvas={xshift=-2pt,yshift=-2pt}] (3) -- (t) ;
    \draw [thick, arrows={-Stealth[length=5pt,width=5pt]}, transform canvas={xshift=1pt,yshift=1pt}] (1) -- (t) ;
        \end{tikzpicture}
        }
  \end{tabular}
  \caption{Consider subnetworks of (a) under mass-action kinetics, whose associated dynamics is given by \eqref{eq:tri}. If the coefficient of $\xx^{\yy_1}$ in $\dot{x}$ is positive and the coefficient of $\xx^{\yy_3}$ in $\dot{x}$ is negative, then the system  \eqref{eq:tri} can be realized by a single-target network, determined by the sign of $\xx^{\yy_2}$ in $\dot{x}$. If the net directions are as shown in (b), then \eqref{eq:tri} can be realized by the single-target network in (c). Similarly, if the net directions appear as in (d), then \eqref{eq:tri} can be realized by the network in (e). 
  }
  \label{fig:tri}
\end{figure}
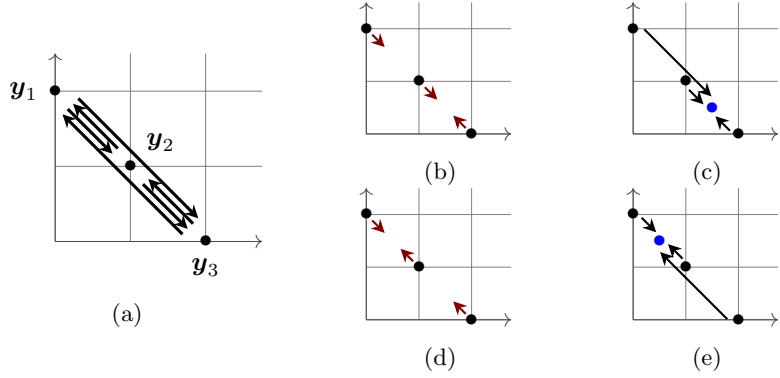

\begin{ex}\label{ex:trionline}
    Consider the complete graph on the vertices 
    \eq{ 
        \yy_1 = \begin{pmatrix} 0 \\ 2 \end{pmatrix},
        \quad 
        \yy_2 = \begin{pmatrix} 1 \\ 1 \end{pmatrix},
        \quad 
        \yy_3 = \begin{pmatrix} 2 \\ 0 \end{pmatrix},
    }
    as shown in \Cref{fig-tri-full}. In \cite{disguised_toric}, this network under mass-action kinetics, was shown to be dynamically equivalent to a complex-balanced system for any vector of positive rate constants. We claim that under mass-action kinetics, any subnetwork for which $\yy_1 = (0,2)^\top$ and $\yy_3 = (2,0)^\top$ are sources, can be realized by a single-target network, and is dynamically equivalent to a detailed-balanced system.

    Let $\kk_{ij} \geq 0$ be the rate constant (if non-zero) of the edge $\yy_i \to \yy_j$. The associated dynamical system 
    \eqn{\label{eq:tri} \begin{array}{l}\ds 
        \frac{dx}{dt} = \hphantom{-{}}y^2(\kk_{12} + 2 \kk_{13} )
            + xy(-\kk_{21} + \kk_{23})
            + x^2 (-\kk_{32} -2\kk_{31})
        \\[8pt] \ds 
        \frac{dy}{dt} = -y^2(\kk_{12} + 2 \kk_{13} )
            - xy(-\kk_{21} + \kk_{23})
            - x^2 (-\kk_{32} -2\kk_{31})
        \end{array}
    }
    is a homogeneous degree two polynomial system, where we assume $\kk_{12} +  \kk_{13} > 0$ and $\kk_{32} + \kk_{31} > 0$. The sign of $-\kk_{12}+\kk_{23}$ determines the structure of the single-target network that the mass-action system is dynamically equivalent to. Consider the net direction from each vertex, given by the weighted sum of reaction vectors originating from that vertex with weights given by the rate constants. If $-\kk_{12}+\kk_{23} \geq 0$, the net direction from each vertex is shown in \Cref{fig-triA1} (with possibly nothing from $\yy_2$). Then the system can be realized by the single-target network in \Cref{fig-triA2}. Denote by $\kk'_{i}$ the rate constant from $\yy_i$ to the target $(1.5,0.5)^\top$; the rate constants for the system on \Cref{fig-triA2} are  
        \eq{ 
            \kk'_1 = \frac{2}{3} (\kk_{12} + 2\kk_{13}), 
            \quad 
            \kk'_2 = 2(-\kk_{12} + \kk_{23}),
            \quad 
            \kk'_3 = 2(\kk_{32} + 2 \kk_{31}). 
        }
    A similar argument shows that if $-\kk_{12} + \kk_{23} < 0$, the net direction from each vertex is shown in \Cref{fig-triB1}. The system can be realized by the single-target network in \Cref{fig-triB2}, with rate constants 
        \eq{ 
            \kk'_1 = 2 (\kk_{12} + 2\kk_{13}), 
            \quad 
            \kk'_2 = 2(\kk_{12} - \kk_{23}),
            \quad 
            \kk'_3 =  \frac{2}{3}(\kk_{32} + 2 \kk_{31}).
        }
    This follows by considering linear equations coming from each of the three source vertices. For example, at $\yy_1$, dynamical equivalence dictates that 
        \eq{ 
            \kk_{12} \begin{pmatrix*}[r] 1 \\ -1 \end{pmatrix*} + \kk_{13} \begin{pmatrix*}[r] 2 \\ -2 \end{pmatrix*} = \kk'_1 \begin{pmatrix*}[r] 3/2 \\ -3/2 \end{pmatrix*}, 
        }
    from which one can easily solve for $\kk'_1$. Similar considerations at the other source vertices provide the remaining rate constants.

    This example can be extended to homogeneous polynomials of two variables. Order the terms of such a polynomial $p(x,y)$ by ascending degree of $x$. If the coefficient of the first term is positive, coefficient for the last term is negative, and there is exactly one sign change between consecutive terms, then for any positive initial condition, the system 
        \eq{ 
            \frac{dx}{dt} = p(x,y) \quad \text{and} \quad  \frac{dy}{dt} = -p(x,y)
        }
    has exactly one positive steady state, which is globally stable. Indeed, the system can be realized by a single-target network, and is dynamically equivalent to a detailed-balanced system. 
\end{ex}

\begin{ex}
    In this example, we consider non-linear dynamical systems on $\rrpp^n$ of the form
    \eqn{ \label{eq:decay}
        \frac{d\xx}{dt} = \sum_{i=1}^m - \kk_i \xx^{\yy_i} \yy_i,
    }
    where $\kk_i > 0$ and $\yy_i \in \rr^n$ such that the origin is a positive convex combination of $\{\yy_i \st i = 1,2,\ldots, m\}$. For example, 
    \eq{ 
        \frac{d}{dt} \begin{pmatrix} x \\ y \\z \end{pmatrix} 
        = \kk_1 x^{-1} y^{-2} \begin{pmatrix} 1 \\ 2 \\ 0 \end{pmatrix} 
        + \kk_2 y^{-3} z^{-1} \begin{pmatrix} 0 \\ 3 \\ 1 \end{pmatrix} 
        + \kk_3 x^{-2} y^3 z^2 \begin{pmatrix} 2 \\ -3 \\ -2 \end{pmatrix} 
        + \kk_4 xy^2 z \begin{pmatrix} -1 \\ -2 \\ -1 \end{pmatrix} 
        + \kk_5 x^4 y^{-2} z^{\frac{3}{2}} \begin{pmatrix} -4 \\ 2 \\ -3/2 \end{pmatrix} 
    }
    belongs to this class. At first sight of the differential equations, there may be very little reason to believe that this system has a unique positive steady state, which is globally stable, within the affine space parallel to $\Span\{ \yy_i \st i=1,2,\ldots, m\}$. However, with the tools developed in this paper, uniqueness of steady states and global stability immediately follow from \Cref{thm:stnmas-stable}. The reaction network that generates \eqref{eq:decay} under mass-action kinetics consists of the reactions $\yy_i \to \vv 0$ with rate constant $\kk_i > 0$. By definition, the unique target $\vv 0$ is in the relative interior of $\{ \yy_i \st i=1,2,\ldots, m\}$. Therefore by \Cref{thm:stnmas}, for any positive initial condition, there is exactly one positive steady state which is globally attracting. 
\end{ex}

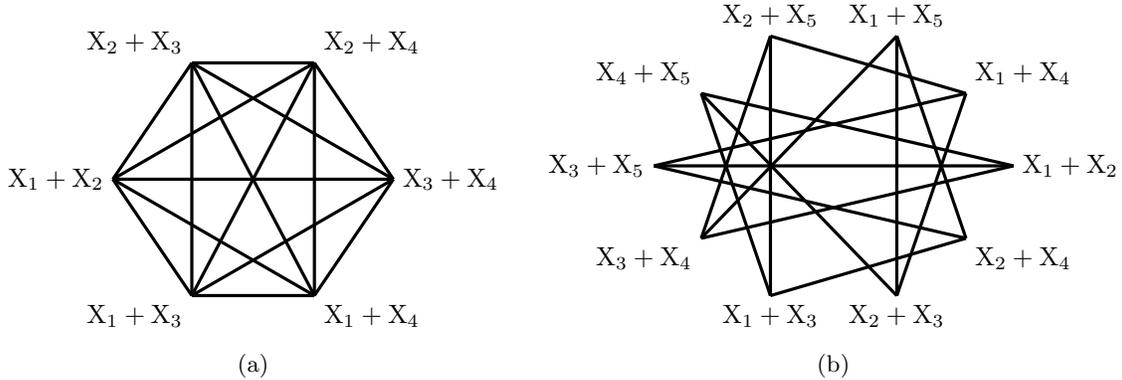
\begin{figure}[h!tbp]
\centering
\begin{subfigure}[b]{0.415\textwidth}
    \centering 
    \begin{tikzpicture}
    \begin{scope}[scale=2.1]
    \node (12) at (1.25,0) {$\cf{X_3+X_4}$}; 
    \node (13) at ({0.5*1.5},{0.866*1}) {$\cf{X_2+X_4}$}; 
    \node (14) at ({0.5*1.5},-{0.866*1}) {$\cf{X_1+X_4}$}; 
    \node (23) at ({-0.5*1.5},{0.866*1}) {$\cf{X_2+X_3}$}; 
    \node (24) at ({-1.25},0) {$\cf{X_1+X_2}$}; 
    \node (34) at ({-0.5*1.5},{-0.866*1}) {$\cf{X_1+X_3}$}; 
    
    \draw [very thick] (13.south west) -- (14.north west);
    \draw [very thick] (13.south west) -- (23.south east);
    \draw [very thick] (13.south west) -- (24.east);
    \draw [very thick] (13.south west) -- (34.north east);
    
    \draw [very thick] (14.north west) -- (23.south east);
    \draw [very thick] (14.north west) -- (24.east);
    \draw [very thick] (14.north west) -- (34.north east);
    
    \draw [very thick] (23.south east) -- (24.east);
    \draw [very thick] (23.south east) -- (34.north east);
    \draw [very thick] (24.east) -- (34.north east);
    
    \draw [very thick] (12.west) -- (13.south west);
    \draw [very thick] (12.west) -- (14.north west);
    \draw [very thick] (12.west) -- (23.south east);
    \draw [very thick] (12.west) -- (24.east);
    \draw [very thick] (12.west) -- (34.north east);
\end{scope}
    \end{tikzpicture}
    \caption{}
    \label{fig:dense4}
\end{subfigure}\hspace{0.25cm}
\begin{subfigure}[b]{0.475\textwidth}
    \centering 
    \begin{tikzpicture}
    \begin{scope}[scale=2.1]
    \node (12) at (1.5,0) {$\cf{X_1+X_2}$}; 
    \node (14) at (1.2, 0.5878) {$\cf{X_1+X_4}$}; 
    \node (15) at (0.4, 0.9511) {$\cf{X_1+X_5}$}; 
    \node (25) at (-0.4, 0.9511) {$\cf{X_2+X_5}$}; 
    \node (45) at (-1.2, 0.5878) {$\cf{X_4+X_5}$}; 
    \node (35) at (-1.5,0) {$\cf{X_3+X_5}$}; 
    \node (34) at (-1.2, -0.5878) {$\cf{X_3+X_4}$}; 
    \node (13) at (-0.4, -0.9511) {$\cf{X_1+X_3}$}; 
    \node (23) at (0.4, -0.9511) {$\cf{X_2+X_3}$}; 
    \node (24) at (1.2, -0.5878) {$\cf{X_2+X_4}$}; 
    
    \draw [very thick] (12.west) -- (34.north east);
    \draw [very thick] (12.west) -- (35.east);
    \draw [very thick] (12.west) -- (45.south east);
    
    \draw [very thick] (13.north) -- (24.north west);
    \draw [very thick] (13.north) -- (25.south);
    \draw [very thick] (13.north) -- (45.south east);
    
    \draw [very thick] (14.south west) -- (23.north);
    \draw [very thick] (14.south west) -- (25.south);
    \draw [very thick] (14.south west) -- (35.east);
    
    \draw [very thick] (15.south) -- (23.north);
    \draw [very thick] (15.south) -- (24.north west);
    \draw [very thick] (15.south) -- (34.north east);
    
    \draw [very thick] (23.north) -- (45.south east);
    
    \draw [very thick] (24.north west) -- (35.east);
    
    \draw [very thick] (25.south) -- (34.north east);
\end{scope}
    \end{tikzpicture}
    \caption{}
    \label{fig:dense5}
\end{subfigure}
    \caption{Reversible systems in (a) \Cref{ex:dense4} and (b) \Cref{ex:dense5} that are dynamically equivalent to detailed-balanced systems. Each undirected edge represents a pair of reversible edges.}
    \label{fig:dense}
\end{figure}

\begin{figure}[h!tbp]
\centering
    \begin{tikzpicture}[scale=2.1]
    \coordinate (a) at (-1,-0.75) {};
    \coordinate (b) at (0.74, -1.01) {};
    \coordinate (c) at (0.97,-0.57) {};
    \coordinate (d) at (0.245,0.92) {};
    \draw [gray!50] (a)--(b)--(d)--cycle;
    \draw [gray!50] (b)--(c)--(d)--cycle;
    \draw [gray!50, dashed] (a)--(c);
    \node[outer sep=0pt, inner sep=1pt, red!50!black]  (ad) at (-0.37,0.09) {{$\bullet$}};
    \node[outer sep=0pt, inner sep=1pt] (ab) at (-0.12,-0.88) {{$\bullet$}};
    \node[outer sep=0pt, inner sep=1pt] (ac) at (-0.01,-0.66) {{$\bullet$}};
    
    \node[outer sep=0pt, inner sep=1pt] (bc) at (0.855,-0.79) {{$\bullet$}};
    \node[outer sep=0pt, inner sep=1pt] (bd) at (0.495,-0.045) {{$\bullet$}};
    
    \node[outer sep=0pt, inner sep=1pt] (cd) at (0.6,0.18) {{$\bullet$}};
   \draw [fwdrxn, red!50!black] (ad) -- (ab);
   \draw [fwdrxn, red!50!black] (ad) -- (bd);
   \draw [fwdrxn, red!50!black, dashed] (ad) -- (ac);
   \draw [fwdrxn, red!50!black, dashed] (ad) -- (cd);
   \draw [fwdrxn, red!50!black, dashed, line width=0.6mm] (ad) -- (bc);
   \node [gray!50, above=0pt of d] {$2\cf{X}_1$};
   \node [gray!50, left=0pt of a] {$2\cf{X}_2$}; 
   \node [gray!50, below=3pt of b] {\,\,$2\cf{X}_3$}; 
   \node [gray!50, right=2pt of c] {$2\cf{X}_4$};

   \node [red!50!black, above=0pt of ad] {$\cf{X}_1 + \cf{X}_2${\hspace{0.75cm}}\,};
   \node [below right=-5pt of bd] {$\cf{X}_1 + \cf{X}_3$};
   \node [right=0pt of cd] {$\cf{X}_1 + \cf{X}_4$};
   \node [below=0pt of ab] {$\cf{X}_2 + \cf{X}_3$};
   \node [below right=-5pt of bc] {$\cf{X}_3 + \cf{X}_4$};

    \node [blue] (centre) at (0.2425, -0.35) {\Large{$\bullet$}};
\end{tikzpicture}
    \caption{Geometric argument for dynamically equivalence to single-target network in \Cref{ex:dense4}. Shown are the edges with $\mrm{X}_1 + \mrm{X}_2$ as their source. The centre $\left(\frac{1}{2}, \frac{1}{2}, \frac{1}{2}, \frac{1}{2} \right)^\top$ of the tetrahedron is marked in blue. With rate constants given in the example, the weighted sum of reaction vectors points from the source to the centre.
    }
    \label{fig:dense4-embedded}
\end{figure}
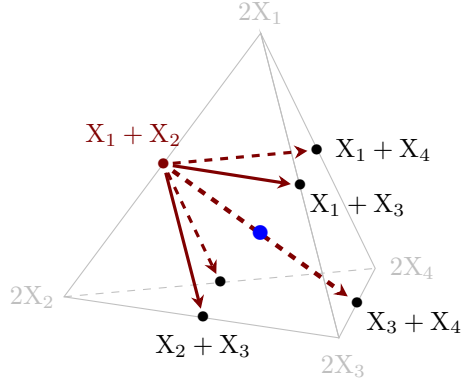

\begin{ex}
\label{ex:dense4}
    Consider the reaction network shown in \Cref{fig:dense4}, where each edge represents a reversible pair of reactions. Then in the language of the standard basis $\{ \hat{\mbf{e}}_i \st i = 1,\ldots 4\}$ of $\rr^4$, the six vertices are $\{ \yy_{ij} = \hat{\mbf{e}}_i + \hat{\mbf{e}}_j \st 1 \leq i < j \leq 4 \}$. Edges take the form $\yy_{ij} \RR \yy_{pq}$ where $(i,j) \neq (p,q)$, with the network $G$ being a complete graph.

    A given vertex $\yy_{ij}$ is source to two kinds of edges: an edge whose target has disjoint support from the source (i.e., $\yy_{ij} \to \yy_{pq}$ where $i$, $j$, $p$, $q$ are distinct integers), and those whose targets share an index with the source (i.e., $\yy_{ij} \to \yy_{iq}$ or $\yy_{ij} \to \yy_{pj}$). The latter represents a chemical reaction with the common species acting as a catalyst.  

    Rate constant of an edge is assigned based on the source vertex and which type of edge it is. For example, consider the source vertex $\yy_{12} = \hat{\mbf{e}}_1 + \hat{\mbf{e}}_2$. The edge $\yy_{12} \to \yy_{34}$ is assigned an arbitrary rate constant $\kk_{12a} > 0$, while the catalytic reactions 
    \eq{ 
        \yy_{12} \to \yy_{13}, \quad 
        \yy_{12} \to \yy_{14}, \quad 
        \yy_{12} \to \yy_{23}, \quad 
        \yy_{12} \to \yy_{24}
    }
    are assigned rate constants $\kk_{12b} > 0$. The rate constant for $\yy_{13} \to \yy_{24}$ is $\kk_{13a} > 0$, while other edges originating from $\yy_{13}$ have rate constants $\kk_{13b} > 0$. The remaining edges are assigned rate constants in a similar manner. 

    The mass-action system $\Gk$ is not detailed-balanced in general since the circuit condition \eqref{eq:circuit} is generically violated along some cycles, e.g., the cycle with vertices $\yy_{12}$, $\yy_{24}$ and $\yy_{34}$ is associated to the condition
    \eq{ 
        \kk_{12b}\kk_{24a}\kk_{34a} = \kk_{12a}\kk_{34b}\kk_{24b}. 
    }
    Moreover, this reversible network has deficiency $\delta = 2$; therefore, the mass-action system is not complex-balanced in general as well. 

    Nonetheless, the system can be realized by a single-target network and is dynamically equivalent to a detailed-balanced system. The weighted sum of reaction vectors coming out of the vertex $\yy_{12}$ is
    \eq{ 
        &\kk_{12a} \left( \yy_{34} - \yy_{12} \right) 
        + \kk_{12b} \left( \yy_{13} - \yy_{12} \right) 
        + \kk_{12b} \left( \yy_{14} - \yy_{12} \right) 
        + \kk_{12b} \left( \yy_{23} - \yy_{12} \right) 
        + \kk_{12b} \left( \yy_{24} - \yy_{12} \right) 
        \\={}&  \kk_{12a} \begin{pmatrix*}[r]
            -1 \\ -1 \\1 \\ 1 
        \end{pmatrix*}  + \kk_{12b} \begin{pmatrix*}[r]
            -2 \\ -2 \\2 \\ 2 
        \end{pmatrix*}  
        = 2\left(\kk_{12a} + 2\kk_{12b} \right)  \begin{pmatrix}
            1/2-1 \\ 1/2-1 \\1/2\hphantom{{}-1} \\ 1/2\hphantom{{}-1}
        \end{pmatrix}  ,
    }
    which is also the weighted reaction vector of $\yy_{12} \to \left( \frac{1}{2}, \frac{1}{2}, \frac{1}{2}, \frac{1}{2}\right)^\top$ with rate constant $2\left(\kk_{12a} + 2\kk_{12b} \right)$. See \Cref{fig:dense4-embedded} for the geometry of this calculation. By symmetry, the weighted sum of reaction vectors out of any vertex of $G$ can be written as an edge to $\left( \frac{1}{2}, \frac{1}{2}, \frac{1}{2}, \frac{1}{2}\right)^\top$. Therefore, the mass-action system generated by the network in \Cref{fig:dense4} is dynamically equivalent to a single-target network with target vertex $\left( \frac{1}{2}, \frac{1}{2}, \frac{1}{2}, \frac{1}{2}\right)^\top$, which is in the relative interior of the Newton polytope. By \Cref{thm:stnmas}, the mass-action system is dynamically equivalent to a globally stable detailed-balanced system.
\end{ex}

\begin{ex}
\label{ex:dense5}
    Consider the reversible reaction network in \Cref{fig:dense5} in $\rr^5$. This network is similar to that of \Cref{ex:dense4} except it has \emph{no} catalytic reaction. The ten vertices are $\cf{X}_i + \cf{X}_j$ with $1 \leq i < j \leq 5$. Edges take the form $\cf{X}_i + \cf{X}_j \RR \cf{X}_p + \cf{X}_q$ where $i$, $j$, $p$, $q$ are all distinct. Further assume the rate constants depend only on the source vertices, i.e., all reactions originating from the vertex $\cf{X}_i + \cf{X}_j$ have the same rate constant. 
    
    This reversible network has deficiency $\delta = 5$. This system is in general neither complex-balanced nor detailed-balanced. For example, the Wegscheider's condition involves the equation 
        \eq{ 
            \kk_{13}\kk_{24}\kk_{35}=\kk_{12}\kk_{34}
        }
    among many others. Nonetheless, the system can be realized by a single-target network, whose target vertex $\frac{2}{5}\cf{X_1} + \frac{2}{5}\cf{X_2} + \frac{2}{5}\cf{X_3} + \frac{2}{5}\cf{X_4} + \frac{2}{5}\cf{X_5}$ lies in the relative interior of the Newton polytope. Therefore, the mass-action system is dynamically equivalent to a globally stable detailed-balanced system. 
\end{ex}

\section{Networks with two targets}
\label{sec:multtarg}

In the previous section, we have characterized the dynamics of all single-target networks under mass-action kinetics. In particular, we have seen that if the target vertex is in the relative interior of the Newton polytope, then any mass-action system generated by that network is dynamically equivalent to a detailed-balanced system, which has a globally attracting positive steady state within each stoichiometric compatibility class.

One may wonder if a similar result holds for networks with multiple targets, each in the relative interior of the Newton polytope. Networks with such \lq\lq inward pointing\rq\rq\ reaction vectors, or \emph{endotatic networks}~\cite{CraciunNazarovPantea2013}, are conjectured to be \emph{persistent}, i.e., trajectories are bounded away from the boundary, and even \emph{permanent}, i.e., admit a globally attracting compact set within each stoichiometric compatibility class. These conjectures have been proved for certain classes of networks: weakly reversible networks with one connected component~\cites{Anderson2011_GAC, BorosHofbauer2019}, strongly endotactic networks~\cite{GopalkrishnanMillerShiu2014}, and two-dimensional networks~\cite{CraciunNazarovPantea2013,Pantea2012}.

Even with just \emph{two} target vertices, there exist strongly endotactic networks with multiple positive steady states (within the same stoichiometric compatibility class), and thus cannot be globally stable. In the following examples, we relax the requirement, instead searching for a dynamically equivalent \emph{complex-balanced} system. Because vertices that do not appear explicitly as monomials in the differential equations are not necessary when searching for a dynamically equivalent complex-balanced system~\cite{CraciunJinYu2019}, we restrict our attention to subnetworks on the complete graph defined by the source vertices.

\begin{ex}
\label{ex:2targets1D}
For example, consider the mass-action system 
    \tikzc{
    \draw [step=1.5, gray, very thin] (0,-0.5) grid (8.5,0.5);
        \node (1) at (0,0) {$\bullet$};
        \node (2) at (3,0) {$\bullet$};
        \node (3) at (4.5,0) {$\bullet$};
        \node (4) at (7.5,0) {$\bullet$};
        \node (a) at (1.5, 0) {\blue{$\bullet$}};
        \node (b) at (6, 0) {\blue{$\bullet$}};
            \node at (0,0) [left] {$\yy_1$\,};
            \node at (3,0) [right] {\,$\yy_2$};
            \node at (4.5,0) [left] {$\yy_3$\,};
            \node at (7.5,0) [right] {\,$\yy_4$};
        \draw [fwdrxn, transform canvas={yshift=0pt}] (1) -- (a) node [midway, above] {$\kk_1$};
        \draw [fwdrxn, transform canvas={yshift=0pt}] (2) -- (a) node [midway, above] {$\kk_2$};
        \draw [fwdrxn, transform canvas={yshift=0pt}] (3) -- (b) node [midway, above] {$\kk_3$};
        \draw [fwdrxn, transform canvas={yshift=0pt}] (4) -- (b) node [midway, above] {$\kk_4$};
    }
where the source vertices are $\yy_1 = 0$, $\yy_2 = 2$, $\yy_3 = 3$ and $\yy_4 = 5$, and the target vertices are $\yy_5 = 1$ and $\yy_6 = 4$. The associated dynamical system 
    \eq{ 
        \frac{dx}{dt} &= \kk_1 - \kk_2 x^2 + \kk_3 x^3 - \kk_4 x^5
    }
has \emph{multiple} positive steady states for some choice of $\kk_i > 0$ by Descartes' rule of signs. In particular, for these choices of $\kk_i > 0$, it cannot be dynamically equivalent to a detailed-balanced (or complex-balanced) system, which necessarily has a \emph{unique} positive steady state.

We claim that this system $\Gk$ is dynamically equivalent to a complex-balanced system if and only if $\kk_1 \kk_4 \geq \kk_2 \kk_3$. Let $G'$ be the complete graph on the source vertices and let $\kk'_{ij} \geq 0$ be the label on $\yy_i \to \yy_j$. The objective is a subgraph of $G'$.

Since there are four sources in $G$, which are also sources in $G'$, there are four non-trivial linear relations on the edge labels of $G$ and those of $G'$, and two trivial equations ($0 = 0$) coming from the vertices $\yy_5$ and $\yy_6$. For example, the dynamical equivalence relation at $\yy_2$ is 
    \eq{ 
        - \kk_2 = -2 \kk_{21} + \kk_{23} + 3 \kk_{24}.
    }
As noted after \Cref{def:flux}, the dynamical equivalence relation can be transformed to that of the fluxes, by multiplying both sides of the linear equation by the source vertex's monomial. Fix a state $x > 0$, the value yet to be determined. Let $J_i = \kk_i x^{\yy_i}$ be the flux across the edge originating from $\vv y_i$. Similarly, on $G'$ let $Q_{ij} = \kk'_{ij} x^{\yy_i}$ be the flux across the edge $\yy_i \to \yy_j$. Then the dynamical equivalence relation at $\yy_2$ can be written as 
    \eq{ 
        - J_2 = -2 Q_{21} + Q_{23} + 3 Q_{24}. 
    }

The switch from rate constants to fluxes is convenient when considering dynamical equivalence and complex-balancing simultaneously. For example, in the objective system $G'_{\vv\kk'}$, the state $x$ is complex-balanced if and only if 
    \eq{ 
        (\kk'_{21} + \kk'_{23} + \kk'_{24}) x^{\yy_2} = \kk'_{12}x^{\yy_1} + \kk'_{32}x^{\yy_3} + \kk'_{42}x^{\yy_4}.
    }
In the language of flux, this reads $ Q_{21} + Q_{23} + Q_{24} = Q_{12} + Q_{32} + Q_{42}$.

Hence, the four dynamical equivalence relations, in terms of fluxes, are 
    \eq{ 
        &\hphantom{{}-{}}J_1 = 2 Q_{12} + 3 Q_{13} + 5 Q_{14}, \\
        &-J_2 = -2 Q_{21} + Q_{23} + 3 Q_{24}, \\
        &\hphantom{{}-{}}J_3 = -3 Q_{31} - Q_{32} + 2Q_{34} , \\
        &-J_4 = -5Q_{41} - 3Q_{42} - 2 Q_{43},
    }
while the complex-balanced conditions on $G'$ are
    \eqn{ 
        Q_{12} + Q_{13} + Q_{14} &= Q_{21} + Q_{31} + Q_{41}, \label{eq:exCB}\\
        Q_{21} + Q_{23} + Q_{24} &= Q_{12} + Q_{32} + Q_{42}, \nonumber \\
        Q_{31} + Q_{32} + Q_{34} &= Q_{13} + Q_{23} + Q_{43}, \nonumber \\
        Q_{41} + Q_{42} + Q_{43} &= Q_{14} + Q_{24} + Q_{34}. \nonumber
    }
    
It is not difficult to see from \eqref{eq:exCB} that the left-hand side of the complex-balanced condition at $\yy_1$ can be rewritten as
    \eq{ 
    Q_{12} + Q_{13} + Q_{14} = 
        \left( \frac{1}{2} J_1 - \frac{3}{2} Q_{13} - \frac{5}{2} Q_{14} \right) + Q_{13} + Q_{14} \leq \frac{1}{2} J_1,
    }
while the right-hand side is 
    \eq{ 
    Q_{21} + Q_{31} + Q_{41} = 
        \left( \frac{1}{2}J_2 + \frac{1}{2} Q_{23} + \frac{3}{2} Q_{24} \right) + Q_{31} + Q_{41}
        \geq \frac{1}{2} J_2. 
    }
In other words, $J_1 \geq J_2$. Similarly, from the last complex-balanced condition, we obtain $J_4 \geq J_3$. Therefore,
$\Gk$ being dynamically equivalent to a complex-balanced system implies $J_1J_4 \geq J_2 J_3$; equivalently, $\kk_1\kk_4 \geq \kk_2 \kk_3$. 

Conversely, suppose $\kk_1\kk_4 \geq \kk_2 \kk_3$ or $J_1J_4 \geq J_2J_3$. Clearly the system of differential equations admits at least one positive steady state $x>0$. At $x$, the fluxes are balanced, i.e., $J_1 + J_3 = J_2 + J_4$. Substituting the inequality into $\kk_1 + \kk_3 x^3 = \kk_2x^2+\kk_4x^5$, we see that $J_1 \geq J_2$ and $J_4 \geq J_3$. The choice  
    \eq{ 
        Q_{12} = Q_{21} = \frac{J_2}{2}, \qquad 
        Q_{34} = Q_{43} = \frac{J_3}{2}, \qquad 
        Q_{14} = \frac{J_1-J_2}{5} \qquad \text{and} \qquad 
        Q_{41} = \frac{J_4-J_3}{5}
    } 
satisfies the dynamical equivalence and complex-balanced conditions. Choose rate constants on the new network to be 
    \eq{ 
        \tilde{\kk}_{12} = \frac{\kk_2}{2}, \qquad 
        \tilde{\kk}_{34} = \frac{\kk_3 x^3}{2}, \qquad 
        \tilde{\kk}_{12} = \frac{\kk_2x^2}{2}, \qquad  
        \tilde{\kk}_{43} = \frac{\kk_3x^3}{2}, \qquad \\
        \tilde{\kk}_{14} = \frac{1}{5} \left( \kk_1 - \frac{\kk_2x^2}{2} \right)\qquad \text{and} \qquad 
        \tilde{\kk}_{41} = \frac{1}{5} \left( \kk_4 - \frac{\kk_3x^{-2}}{2} \right).
        \hspace{0.8cm}
    }
Note that $\tilde{\kk}_{14} = \frac{1}{10}(2J_1 - J_2) > 0$, and $\tilde{\kk}_{41} = \frac{1}{10x^5}(2J_4 - J_3) > 0$. With this choice of rate constants $\vv{\tilde{\kk}}$, the mass-action system is dynamically equivalent to a complex-balanced system with steady state $x$, which is the unique positive steady state of the system.
\end{ex}

\begin{ex}\label{ex:2targets2D}
Consider the network under mass-action kinetics in \Cref{fig:2targets2Da} with four source vertices and two target vertices in $\rr^2$. The source vertices correspond to the monomials $1$, $x^3$, $x^3y^2$ and $y^2$. We will show that this system is dynamically equivalent to a complex-balanced system if and only if 
    \eq{ 
        \frac{1}{25} \leq \frac{\kk_1\kk_3}{\kk_2\kk_4} \leq 25. 
    }
Being dynamically equivalent to a complex-balanced system, if it can be done at all, can be achieved using only the source vertices~\cite{CraciunJinYu2019}; thus we look for a subnetwork of the \emph{complete graph} shown in \Cref{fig:2targets2Db}.

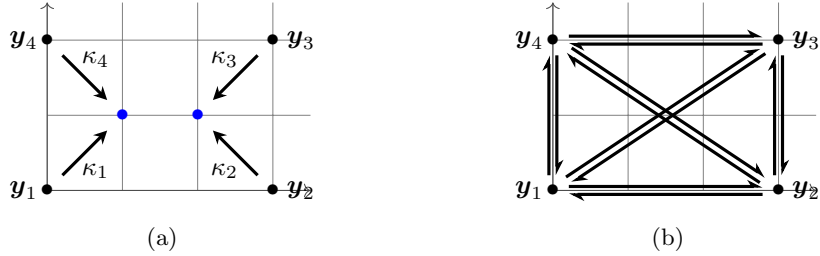
\begin{figure}[h!tbp]
\centering
\begin{subfigure}[b]{0.4\textwidth}
    \centering 
    \begin{tikzpicture}
        \draw [step=1, gray, very thin] (0,0) grid (3.5,2.5);
    \draw [ ->, black!70!white] (0,0)--(3.5,0);
    \draw [ ->, black!70!white] (0,0)--(0,2.5);
    
    \node (c1) at (1,1) {\blue{$\bullet$}};
    \node (c2) at (2,1) {\blue{$\bullet$}};
    
    \node (1) at (0,0) {$\bullet$}; 
    \node (2) at (3,0) {$\bullet$}; 
    \node (3) at (3,2) {$\bullet$}; 
    \node (4) at (0,2) {$\bullet$}; 
    
    \node at (0,0) [left] {$\yy_1$};
    \node at (0,2) [left] {$\yy_4$};
    \node at (3,0) [right] {\,$\yy_2$};
    \node at (3,2) [right] {\,$\yy_3$};
    
    \draw [fwdrxn, transform canvas={yshift=0pt}] (1) -- (c1) node [midway, below right] {\!\!\!$\kk_1$};
    \draw [fwdrxn, transform canvas={yshift=0pt}] (4) -- (c1) node [midway, above right] {\!\!\!$\kk_4$};
    \draw [fwdrxn, transform canvas={yshift=0pt}] (2) -- (c2) node [midway, below left] {$\kk_2$\!\!\!};
    \draw [fwdrxn, transform canvas={yshift=0pt}] (3) -- (c2) node [midway, above left] {$\kk_3$\!\!\!};
    \end{tikzpicture}
    \caption{}
    \label{fig:2targets2Da}
\end{subfigure}
\begin{subfigure}[b]{0.4\textwidth}
    \centering 
    \begin{tikzpicture}
        \draw [step=1, gray, very thin] (0,0) grid (3.5,2.5);
    \draw [ ->, black!70!white] (0,0)--(3.5,0);
    \draw [ ->, black!70!white] (0,0)--(0,2.5);

    \node (1) at (0,0) {$\bullet$}; 
    \node (2) at (3,0) {$\bullet$}; 
    \node (3) at (3,2) {$\bullet$}; 
    \node (4) at (0,2) {$\bullet$}; 
    
    \node at (0,0) [left] {$\yy_1$};
    \node at (0,2) [left] {$\yy_4$};
    \node at (3,0) [right] {\,$\yy_2$};
    \node at (3,2) [right] {\,$\yy_3$};
    
    \draw [revrxn, transform canvas={yshift=1.5pt}] (1) -- (2);
    \draw [revrxn, transform canvas={yshift=-1.5pt}] (2) -- (1);
    
    \draw [revrxn, transform canvas={yshift=1.5pt}] (4) -- (3);
    \draw [revrxn, transform canvas={yshift=-1.5pt}] (3) -- (4);

    \draw [revrxn, transform canvas={xshift=1.5pt}] (4) -- (1);
    \draw [revrxn, transform canvas={xshift=-1.5pt}] (1) -- (4);
    
    \draw [revrxn, transform canvas={xshift=1.5pt}] (3) -- (2);
    \draw [revrxn, transform canvas={xshift=-1.5pt}] (2) -- (3);
    
    \draw [revrxn, transform canvas={xshift=-1.06pt, yshift=1.06pt}] (1) -- (3);
    \draw [revrxn, transform canvas={xshift=1.06pt, yshift=-1.06pt}] (3) -- (1);
    
    \draw [revrxn, transform canvas={xshift=1.06pt, yshift=1.06pt}] (4) -- (2);
    \draw [revrxn, transform canvas={xshift=-1.06pt, yshift=-1.06pt}] (2) -- (4);
    \end{tikzpicture}
    \caption{}
    \label{fig:2targets2Db}
\end{subfigure}
    \caption{(a) The mass-action system with two target vertices from \Cref{ex:2targets2D}, which is dynamically equivalent to a complex-balanced system using a subnetwork of (b) if and only if $\frac{1}{25} \leq \frac{\kk_1\kk_3}{\kk_2\kk_4} \leq 25$. }
    \label{fig:2targets2D}
\end{figure}

Suppose $\vv x > \vv 0$ is a steady state for which the system in \Cref{fig:2targets2Da} is dynamically equivalent to a complex-balanced system. Note that $\vv x$ is a positive steady state for the system. Let $J_i = \kk_i \xx^{\yy_i} > 0$ define a flux on the network. The steady state flux $\vv J$ thus satisfies $J_1 = J_3$ and $J_2 = J_4$. Let $\vv Q$, where $Q_{ij} \geq 0$ is to be determined, denote the flux across the edge $\yy_i \to \yy_j$ in \Cref{fig:2targets2Db}. Dynamical equivalence is obtained if and only if 
    \eq{ 
        J_1 \begin{pmatrix} 1 \\ 1 \end{pmatrix} = 
        \begin{pmatrix} 
            3 Q_{12} + 3 Q_{13} \\
            2 Q_{14} + 2 Q_{13}
        \end{pmatrix}, \quad \!
        \qquad 
        & J_3 \begin{pmatrix} -1 \\ -1 \end{pmatrix} = 
        \begin{pmatrix} 
            -3 Q_{31} - 3 Q_{34} \\
            -2 Q_{31} - 2 Q_{32}
        \end{pmatrix},
        \\
        J_2 \begin{pmatrix} -1 \\ 1 \end{pmatrix} = 
        \begin{pmatrix} 
            -3 Q_{21} - 3 Q_{24} \\
            2 Q_{23} + 2 Q_{24}
        \end{pmatrix}, 
        \qquad 
        & J_4 \begin{pmatrix} 1 \\ -1 \end{pmatrix} = 
        \begin{pmatrix} 
            3 Q_{42} + 3 Q_{43} \\
            -2 Q_{42} - 2 Q_{41}
        \end{pmatrix}.
    }
Meanwhile, complex-balancing is obtained on a subnetwork of that in \Cref{fig:2targets2Db} if and only if  
    \eqn{ 
        Q_{12} + Q_{13} + Q_{14} &= Q_{21} + Q_{31} + Q_{41}, \label{eq:ex:CB2} \\
        Q_{21} + Q_{23} + Q_{24} &= Q_{12} + Q_{32} + Q_{42}, \label{eq:ex:CB3} \\
        Q_{31} + Q_{32} + Q_{34} &= Q_{13} + Q_{23} + Q_{43}, \nonumber \\
        Q_{41} + Q_{42} + Q_{43} &= Q_{14} + Q_{24} + Q_{34}. \nonumber
    }
    
Consider \eqref{eq:ex:CB2}. The left-hand side satisfies the inequality
    \eq{ 
        Q_{12} + Q_{13} + Q_{14} &= \left( Q_{12} + Q_{13} \right) + \left( Q_{14} + Q_{13} \right) - Q_{13}
        = \frac{1}{3} J_1 + \frac{1}{2} J_1 - Q_{13} \leq \frac{5}{6} J_1.
    } 
By substituting the dynamical equivalence equation $J_4 = 2Q_{42} + 2Q_{41}$, we see that the right-hand side is 
    \eq{ 
        Q_{21} + Q_{31} + Q_{41} 
        \geq Q_{41} =  \frac{1}{2} J_4 - Q_{42} 
        = \frac{1}{6} J_4 + \frac{1}{3} J_4 - Q_{42} 
        = \frac{1}{6} J_4 + Q_{43}
        \geq \frac{1}{6} J_2.
    }
Hence for the system in \Cref{fig:2targets2Da} to be dynamically equivalent to a complex-balanced one, we have $J_2 \leq 5 J_1$. Similarly, from \eqref{eq:ex:CB3}, it can be shown that $J_1 \leq 5 J_2$. Therefore, complex-balancing on a subnetwork of  \Cref{fig:2targets2Db} implies $\frac{1}{5} \leq \frac{J_1}{J_2} \leq 5$. Since $J_1 = J_3$ and $J_2 = J_4$ at steady state, i.e., $\kk_1 = \kk_3 x^3 y^2$ and $\kk_2 x^3 = \kk_4 y^2$ by definition, so
    \eqn{ \label{eq:2target-ineq}
        \frac{1}{25} \leq  \frac{J_1J_3}{J_2J_4} = \frac{(\kk_1)(\kk_3x^3y^2)}{(\kk_2x^3)(\kk_4y^2)} \leq 25. 
    }
It follows that 
    \eq{ 
        \frac{1}{25} \leq \frac{\kk_1\kk_3}{\kk_2\kk_4} \leq 25
    }
is a necessary condition for dynamical equivalence to complex-balancing.

Next we show that the inequality $\frac{1}{25} \leq \frac{\kk_1\kk_3}{\kk_2\kk_4} \leq 25$ is sufficient for the system in \Cref{fig:2targets2Da} to be dynamically equivalent to a complex-balanced system on a subnetwork of \Cref{fig:2targets2Db}. We first deduce from \eqref{eq:2target-ineq} that a positive steady state exists, i.e., that there exist $x^3$ and $y^2 > 0$ such that $\kk_1 = \kk_3 x^3 y^2$ and $\kk_2 x^3 = \kk_4 y^2$. It is not difficult to see that 
        \eq{ 
        x^3 =  \sqrt{\frac{\kk_1\kk_4}{\kk_2\kk_3}}
            \quad \text{and} \quad 
        y^2 = \sqrt{\frac{\kk_1\kk_2}{\kk_3\kk_4}}
        }
is a solution. Defining the fluxes to be $J_1 = \kk_1$, $J_3 = \kk_3 x^3 y^2$, $J_2 = \kk_2 x^3$ and $J_4 = \kk_4 y^2$, we obtain the inequality $\frac{1}{5} \leq \frac{J_1}{J_2} \leq 5$. Moreover, $J_1 = J_3$ and $J_2 = J_4$.

When $J_1 = 5J_2$, a solution $\vv Q \geq \vv 0$ to the dynamical equivalence and complex-balanced conditions above is 
    \eq{ 
        Q_{14} = Q_{32} = \frac{J_1}{6} , \qquad & Q_{13} = Q_{31} = \frac{J_1}{3} , \\
        Q_{41} = Q_{23} = \frac{J_2}{2} , \qquad & Q_{21} = Q_{43} = \frac{J_2}{3} , 
    }
and the remaining $Q_{ij} = 0$. 
When $5J_1 = J_2$, a solution $\vv Q \geq \vv 0$ to the dynamical equivalence and complex-balanced conditions above is 
    \eq{ 
        Q_{41} = Q_{23} = \frac{J_2}{6}, \qquad & Q_{24}=Q_{42} = \frac{J_2}{3}, \\
        Q_{14} = Q_{32} = \frac{J_1}{2}, \qquad & Q_{12} = Q_{34} = \frac{J_1}{3},
    }
and the remaining $Q_{ij} = 0$. Whenever $\frac{1}{5} < \frac{J_1}{J_2} < 5$, the system is a convex combination of the two extremal cases, with a solution given by the appropriate convex combination of the two systems in \Cref{fig:2targets2D-bdycase}. Therefore, when $\frac{1}{5} \leq \frac{J_1}{J_2} \leq 5$, there exists $\vv Q \geq \vv 0$ satisfying the dynamical equivalence and complex-balanced conditions.

The vector $\vv Q$, which depends on $J_1$, $J_2$ and hence a function of $\vv x$, can be used to generate rate constants for the new network. We first illustrate the process using the case $J_1 = 5J_2$, i.e., when $\kk_2\kk_4 = 25 \kk_1\kk_3$, before we comment on the general case. Let $\kk'_{ij} \geq 0$, the value to be determined, denote the rate constant of the reaction $\yy_i \to \yy_j$. Consider $Q_{21} = \frac{J_2}{3}$ and $Q_{23} = \frac{J_2}{6}$. By definition $J_2 = \kk_2 x^3$ and $Q_{2j} = \kk'_{2j} x^3$. Looking at the equation
    \eq{ 
        \kk'_{21} x^3 = Q_{21} = \frac{J_2}{3} = \frac{\kk_2 x^3}{3} ,
    }
it is clear we should choose $\kk'_{21} = \frac{\kk_2}{3}$. A similar argument forces $\kk'_{23} = \frac{\kk_2}{2}$. Noting that $J_3 = J_1$ and $J_4 = J_2$, we conclude $\kk'_{32} = \frac{\kk_3}{6}$, $\kk'_{31} = \frac{\kk_3}{3}$, $\kk'_{41} = \frac{\kk_4}{2}$, and $\kk'_{43} = \frac{\kk_4}{3}$. The mass-action system is shown in \Cref{fig:2targets2Dc}. 

In the general case $\frac{1}{25} \leq \frac{\kk_1\kk_3}{\kk_2\kk_4} \leq 25$, as noted earlier a solution $\vv Q \geq \vv 0$ exists satisfying the dynamical equivalence and complex-balancing conditions. Such $\vv Q$ is a convex combination of the extremal cases; as a result, for any $j$, it is always the case that $Q_{1j}$ and $Q_{3j}$ are fractions of $J_1 = J_3$, and $Q_{2j}$ and $Q_{4j}$ are fractions of $J_2 = J_4$. Since $Q_{ij}$ and $J_i$ are both scalar multiples of $\xx^{\yy_i}$, the monomial gets cancelled from the equation and one can solve for $\kk'_{ij}$ as a fraction of $\kk_{ij}$. 
\end{ex}

\begin{figure}[h!tbp]
\centering
\begin{subfigure}[b]{0.4\textwidth}
    \centering 
    \begin{tikzpicture}
        \draw [step=1, gray, very thin] (0,0) grid (3.5,2.5);
    \draw [ ->, black!70!white] (0,0)--(3.5,0);
    \draw [ ->, black!70!white] (0,0)--(0,2.5);

    \node (1) at (0,0) {$\bullet$}; 
    \node (2) at (3,0) {$\bullet$}; 
    \node (3) at (3,2) {$\bullet$}; 
    \node (4) at (0,2) {$\bullet$}; 
    
    \node at (0,0) [left] {$\yy_1$};
    \node at (0,2) [left] {$\yy_4$};
    \node at (3,0) [right] {\,$\yy_2$};
    \node at (3,2) [right] {\,$\yy_3$};
    
    \draw [fwdrxn] (2) -- (1)  node [midway, below] {\footnotesize $\frac{\kk_2}{3}$};
    
    \draw [fwdrxn] (4) -- (3) node [midway, above] {\footnotesize $\frac{\kk_4}{3}$};

    \draw [revrxn, transform canvas={xshift=1.5pt}] (4) -- (1) node [midway, right] {\footnotesize $\frac{\kk_4}{2}$};
    \draw [revrxn, transform canvas={xshift=-1.5pt}] (1) -- (4) node [midway, left] {\footnotesize $\frac{\kk_1}{6}$};
    
    \draw [revrxn, transform canvas={xshift=1.5pt}] (3) -- (2) node [midway, right] {\footnotesize $\frac{\kk_3}{6}$};
    \draw [revrxn, transform canvas={xshift=-1.5pt}] (2) -- (3) node [midway, left] {\footnotesize $\frac{\kk_2}{2}$};
    
    \draw [revrxn, transform canvas={xshift=-1.06pt, yshift=1.06pt}] (1) -- (3) node [midway, above left] {\footnotesize $\frac{\kk_1}{3}$\!\!\!};
    \draw [revrxn, transform canvas={xshift=1.06pt, yshift=-1.06pt}] (3) -- (1)node [midway, below right] {\footnotesize \!\!\!$\frac{\kk_3}{3}$};
    
    \end{tikzpicture}
    \caption{}
    \label{fig:2targets2Dc}
\end{subfigure}
\begin{subfigure}[b]{0.4\textwidth}
    \centering 
    \begin{tikzpicture}
        \draw [step=1, gray, very thin] (0,0) grid (3.5,2.5);
    \draw [ ->, black!70!white] (0,0)--(3.5,0);
    \draw [ ->, black!70!white] (0,0)--(0,2.5);

    \node (1) at (0,0) {$\bullet$}; 
    \node (2) at (3,0) {$\bullet$}; 
    \node (3) at (3,2) {$\bullet$}; 
    \node (4) at (0,2) {$\bullet$}; 
    
    \node at (0,0) [left] {$\yy_1$};
    \node at (0,2) [left] {$\yy_4$};
    \node at (3,0) [right] {\,$\yy_2$};
    \node at (3,2) [right] {\,$\yy_3$};

    \draw [fwdrxn] (1) -- (2)  node [midway, below] {\footnotesize $\frac{\kk_1}{3}$};
    
    \draw [fwdrxn] (3) -- (4) node [midway, above] {\footnotesize $\frac{\kk_3}{3}$};

    \draw [revrxn, transform canvas={xshift=1.5pt}] (4) -- (1) node [midway, right] {\footnotesize $\frac{\kk_4}{6}$};
    \draw [revrxn, transform canvas={xshift=-1.5pt}] (1) -- (4) node [midway, left] {\footnotesize $\frac{\kk_1}{2}$};
    
    \draw [revrxn, transform canvas={xshift=1.5pt}] (3) -- (2) node [midway, right] {\footnotesize $\frac{\kk_3}{2}$};
    \draw [revrxn, transform canvas={xshift=-1.5pt}] (2) -- (3) node [midway, left] {\footnotesize $\frac{\kk_2}{6}$};

    \draw [revrxn, transform canvas={xshift=1.06pt, yshift=1.06pt}] (4) -- (2) node [midway, above right] {\footnotesize \!\!\!$\frac{\kk_4}{3}$};
    \draw [revrxn, transform canvas={xshift=-1.06pt, yshift=-1.06pt}] (2) -- (4) node [midway, below left] {\footnotesize $\frac{\kk_2}{3}$\!\!\!};
    \end{tikzpicture}
    \caption{}
    \label{fig:2targets2Dd}
\end{subfigure}
    \caption{The system in \Cref{fig:2targets2Da} is dynamically equivalent to a complex-balanced system if and only if $\frac{1}{25} \leq \frac{\kk_1\kk_3}{\kk_2\kk_4} \leq 25$. The system is equivalent to (a) when $\kk_2\kk_4 = 25 \kk_1\kk_3$ and (b) when $25\kk_2\kk_4 = \kk_1\kk_3$. For $\frac{1}{25} \leq \frac{\kk_1\kk_3}{\kk_2\kk_4} \leq 25$, the dynamically equivalent system is an appropriate convex combination of (a) and (b).}
    \label{fig:2targets2D-bdycase}
\end{figure}
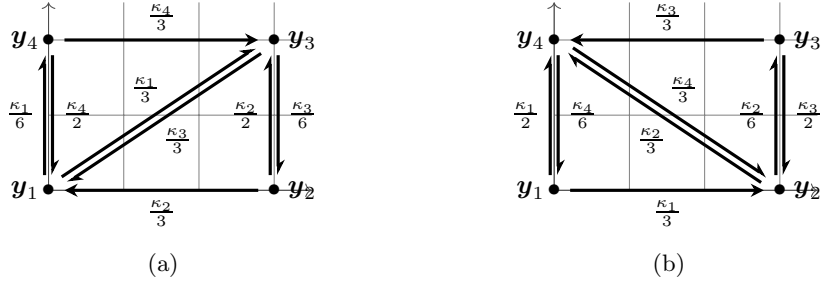

\begin{rmk}
    \Cref{ex:2targets2D} has $(1,1)^\top$ and $(2,1)^\top$ as target vertices, with a distance of $d = 1$ between them. If we consider the two-target network with targets $(a_i,1)$ distance $d > 0$ apart and has midpoint $(1.5,1)^\top$, a similar analysis gives a necessary and sufficient condition for dynamical equivalence to complex-balancing:
    \eq{ 
        \left( \frac{6-d}{d} \right)^2 \geq \frac{\kk_1\kk_3}{\kk_2\kk_4} \geq \left( \frac{d}{6-d}\right)^2. 
    }
    As $d \to 0$, we recover a stable single-target system. As $d \to 3$, the condition becomes $\kk_2\kk_4 = \kk_1\kk_3$, which is necessary and sufficient for the system $\yy_1 \RR \yy_4$, $\yy_2 \RR \yy_3$ to be detailed-balanced~\cite{Feinberg1989}. 
\end{rmk}

Much work has been done to derive algebraic conditions on the rate constants that are necessary and sufficient for complex-balancing~\cites{CraciunDickensteinSturmfelsShiu2009, DickensteinPerezmillan2011, Feinberg1989, disguised_toric}. In \Cref{ex:2targets1D,ex:2targets2D} above, we are able to derive semi-algebraic conditions on the rate constants that are necessary and sufficient for dynamical equivalence to complex-balancing --- at least for networks with special structure. To learn about stability properties of a mass-action system, dynamical equivalence to complex-balancing is extremely informative~\cite{CraciunYu2018}. 

In an upcoming paper, we will prove that for a large class of networks, dynamical equivalence to complex-balancing is an \emph{open condition}, i.e., there exists some open set in parameter space such that rate constants chosen from this set give rise to  mass-action systems that are dynamically equivalent to complex-balanced systems~\cite{CraciunJinYu_OpenCondition}. There is still much more to be done in order to obtain explicit semi-algebraic conditions on rate constants for dynamical equivalence to complex-balancing for general networks. In particular, numerical methods might allow us to  \emph{estimate}  the region in parameter space that gives rise to dynamical equivalence to a complex-balanced system. 


\section{Conclusions}

In this paper we introduced \emph{single-target networks}, and classified the mass-action systems generated by them as either \emph{(i) globally stable} (and actually dynamically equivalent to detailed balanced systems with a single connected component) or \emph{(ii) having no positive steady states} (and moreover having all trajectories converge to the boundary of the positive orthant or to infinity). We showed that these two cases can be differentiated by  a very simple geometric criterion: a single-target mass-action system is globally stable if and only if the target vertex is in the relative interior of the network's Newton polytope. 

%
In general, the single-target condition is quite restrictive, and few networks of interest will satisfy it at the outset. On the other hand, it is a very simple geometric condition, which makes is easy to characterize the set in parameter space where networks that fail to be single-target give rise to systems that {\em can be realized by} single-target networks via dynamical equivalence. Using this idea, we have exhibited several examples where our  results can be useful for analyzing networks that exhibit a high degree of symmetry or geometric structure, even if they are {\em not} single-target networks. 


Finally, recognizing that single-target networks are related to (strongly) endotactic networks, we explored some networks with similar geometry but having multiple targets. For these examples we showed that the corresponding mass-action systems are dynamically equivalent to complex-balanced systems if and only if the rate constants satisfy some semi-algebraic conditions. While our examples have relatively simple structures that allow us to derive explicit inequalities on the rate constants, a natural question and future research direction is whether such semi-algebraic conditions on the rate constants can be obtained for more general classes of reaction networks.

\section*{Acknowledgements} 

This work was supported in part by the National Science Foundation under grant DMS--1816238. P.Y.Y. was also partially supported by NSERC.  

\bibliographystyle{siam}
\bibliography{cit}

\end{document}